\documentclass{amsart}

\usepackage{amsmath}
\usepackage{amsthm}
\usepackage{amssymb}
\usepackage{txfonts}

\theoremstyle{plain}
\newtheorem{definition}{Definition}
\newtheorem{thm}[definition]{Theorem}
\newtheorem{prop}[definition]{Proposition}
\newtheorem{lem}[definition]{Lemma}
\newtheorem{cor}[definition]{Corollary}
\newtheorem{remark}[definition]{Remark}

\def\p#1#2#3#4{{}_{2}\phi_{1}\biggl(\genfrac..{0pt}{}{{#1},\,{#2}}{#3};q,\,#4\biggr)}
\def\P#1#2#3#4#5#6#7#8{{}_{4}\phi_{3}\biggl(\genfrac..{0pt}{}{{#1},\,{#2},\,{#3},\,{#4}}{{#5},\,{#6},\,{#7}};q,\,#8\biggr)}
\def\f#1#2#3#4{{}_{2}\widetilde{\phi}_{1}\biggl(\genfrac..{0pt}{}{{#1},\,{#2}}{#3};#4\biggr)}
\def\y#1#2#3#4#5{y_{#1}\biggl(\genfrac..{0pt}{}{{#2},\,{#3}}{#4};#5\biggr)}
\def\c#1#2#3#4#5#6#7#8{\biggl(\genfrac..{0pt}{}{{#1},\,{#2}}{#3};{#4}\,;\genfrac..{0pt}{}{{#5},\,{#6}}{#7};q,\,#8\biggr)}
\def\Y#1#2#3#4#5#6#7#8{Y\biggl(\genfrac..{0pt}{}{{#1},\,{#2}}{#3};{#4}\,;\genfrac..{0pt}{}{{#5},\,{#6}}{#7};#8\biggr)}
\def\Z#1#2#3#4#5#6#7#8{\widetilde{Y}\biggl(\genfrac..{0pt}{}{{#1},\,{#2}}{#3};{#4}\,;\genfrac..{0pt}{}{{#5},\,{#6}}{#7};#8\biggr)}

\makeatletter
\@addtoreset{equation}{section}

\makeatother

\begin{document}
\title[{\bf Transformation Formulas and Three-Term Relations}]
{{\bf Transformation Formulas and Three-Term Relations for Basic Hypergeometric Series}}
\author[Yuka Suzuki]{Yuka Suzuki}
\date{}

\begin{abstract}
We derive two generalizations of Gasper's transformation formula for basic hypergeometric series. 
Using these generalized formulas, 
we give explicit expressions for the coefficients of three-term relations 
for the basic hypergeometric series ${}_{2} \phi_{1}$, 
which are generalizations of the author's previous results on three-term relations for ${}_{2}\phi_{1}$. 
\end{abstract}

\maketitle

\section{\bf{Introduction}}
In this paper, we derive two generalizations of 
Gasper's~\cite[p.~200, (20)]{Gasper} transformation formula for basic hypergeometric series, 
and using these generalized formulas, we give explicit expressions for the coefficients of 
three-term relations for the basic hypergeometric series ${}_{2} \phi_{1}$. 
These expressions are generalizations of the author's~\cite{Suzuki} previous results. 

The basic hypergeometric series ${}_{r + 1} \phi_{r}$ is defined by 
\begin{align*}
{}_{r + 1}\phi_{r}\biggl(\genfrac..{0pt}{}{a,\, b_{1}, \dotsc, b_{r}}{c_{1}, \dotsc, c_{r}}; q,\, x\biggr) 
&= {}_{r + 1}\phi_{r} (a, b_{1}, \dotsc, b_{r}; c_{1}, \dotsc, c_{r}; q, x) \\
&:= \sum_{i = 0}^{\infty} 
\frac{(a)_{i} (b_{1})_{i} \dotsm (b_{r})_{i}}{(q)_{i} (c_{1})_{i} \dotsm (c_{r})_{i}} x^{i}, \nonumber 
\end{align*}
where $(a)_{i}$ denotes the $q$-shifted factorial defined by 
$(a)_{i} = (a; q)_{i} := (a; q)_{\infty} / (a q^{i}; q)_{\infty}$ with 
$(a)_{\infty} = (a; q)_{\infty} := \prod_{j = 0}^{\infty} (1 - a q^{j})$. 
It is assumed that $\lvert q \rvert < 1$ and 
none of the denominator parameters $c_{1}, \dotsc, c_{r}$ are $1$ or any negative integer power of $q$. 

It is known that for any quadruples of integers $(k, l, m, n)$ and $(k', l', m', n')$, 
the three basic hypergeometric series 
\begin{align*}
\p{a q^{k}}{b q^{l}}{c q^{m}}{x q^{n}}, \quad 
\p{a q^{k'}}{b q^{l'}}{c q^{m'}}{x q^{n'}}, \quad \p{a}{b}{c}{x}
\end{align*}
satisfy a linear relation with coefficients that are rational functions of $a, b, c, q$, and $x$. 
We call such a relation the ``three-term relation (for ${}_{2}\phi_{1}$).'' 

In \cite{Suzuki}, for any integers $k, l$, and $m$, 
the three-term relation 
of the following form is considered: 
\begin{align}\label{3tr}
\p{a q^{k}}{b q^{l}}{c q^{m}}{x} 
= Q_{0} \cdot \p{a q}{b q}{c q}{x} + R_{0} \cdot \p{a}{b}{c}{x}. 
\end{align}
The author obtained explicit expressions for $Q_{0}$ and $R_{0}$ 
using the following transformation formula for basic hypergeometric series obtained by 
Gasper~\cite[p.~200, (20)]{Gasper}: 
If $m, n_{1}, \dotsc, n_{r} \in \mathbb{Z}_{\geq 0}$ and 
$\big\lvert a^{-1} q^{m + 1 - (n_{1} + \dotsm + n_{r})} \big\rvert < 1$, then 
\begin{align}\label{Gasper.tf}
&{}_{r + 2}\phi_{r + 1} \biggl(\genfrac..{0pt}{}{a, \,b, \,c_{1} q^{n_{1}}, \dotsc, c_{r} q^{n_{r}}}{b q^{m + 1}, \,c_{1}, \dotsc, c_{r}}; q,\, a^{-1} q^{m + 1 - (n_{1} + \dotsm + n_{r})}\biggr) \\
&= \frac{(q)_{\infty} (b q / a)_{\infty} (b q)_{m} (c_{1} / b)_{n_{1}} \dotsm (c_{r} / b)_{n_{r}}}{(q / a)_{\infty} (b q)_{\infty} (q)_{m} (c_{1})_{n_{1}} \dotsm (c_{r})_{n_{r}}} 
b^{n_{1} + \dotsm + n_{r} - m} \nonumber \\
&\quad \times {}_{r + 2}\phi_{r + 1}\biggl(\genfrac..{0pt}{}{q^{-m}, \,b, \,b q / c_{1}, \dotsc, b q / c_{r}}{b q / a, \,b q^{1 - n_{1}} / c_{1}, \dotsc, b q^{1 - n_{r}} / c_{r}}; q,\, q\biggr). \nonumber 
\end{align}

In this paper, for any integers $k, l, m$, and $n$, we consider the following three-term relation, 
which is a generalization of $(\ref{3tr})$: 
\begin{align}\label{gen.3tr}
\p{a q^{k}}{b q^{l}}{c q^{m}}{x q^{n}} 
= Q \cdot \p{a q}{b q}{c q}{x} + R \cdot \p{a}{b}{c}{x}. 
\end{align}
Both $(\ref{3tr})$ and $(\ref{gen.3tr})$ are $q$-analogues of 
the three-term relation for the Gauss hyper{\-}geometric series ${}_{2}F_{1}$ 
considered by Ebisu~\cite[p.~256, (1.2)]{Eb1}. 

We derive two generalizations of $(\ref{Gasper.tf})$, 
and using these generalized formulas, we give explicit expressions for $Q$ and $R$. 
It is important to note that the more general three-term relation 
\begin{align*}
\p{a q^{k}}{b q^{l}}{c q^{m}}{x q^{n}} 
= Q' \cdot \p{a q^{k'}}{b q^{l'}}{c q^{m'}}{x q^{n'}} + R' \cdot \p{a}{b}{c}{x} 
\end{align*}
can be derived by eliminating ${}_{2}\phi_{1} (a q, b q; c q; q, x)$ 
from $(\ref{gen.3tr})$ for $(k, l, m, n)$ and $(k', l', m', n')$. 

Throughout this paper, unless explicitly stated otherwise, we assume that 
\begin{align}\label{assump}
a,\, b,\, c,\, \frac{a}{b},\, \frac{c}{a},\, \frac{c}{b} \notin q^{\mathbb{Z}} \cup \left\{0 \right\}. 
\end{align}

\subsection{Main results of this paper}\quad 
Here we present our main results. 
Note that without loss of generality, it is sufficient to consider $(\ref{gen.3tr})$ for only cases satisfying $k \leq l$, 
because ${}_{2}\phi_{1} (a, b; c; q, x)$ is symmetric with respect to the exchange of $a$ and $b$. 

The following theorem asserts the uniqueness of the pair $(Q, R)$ satisfying $(\ref{gen.3tr})$, 
and gives explicit expressions for $Q$ and $R$. 
\begin{thm}\label{gen.main}
For any integers $k, l, m$, and $n$, 
there is a unique pair $(Q, R)$ of rational functions of $a, b, c, q$, and $x$ 
satisfying the three-term relation $(\ref{gen.3tr})$. 
When $k \leq l$, these functions can be expressed as 
\begin{align*}
Q &= Q (k, l, m, n) 
= - \frac{(1 - a) (1 - b) c}{(q - c) (1 - c)} 
\frac{x^{1 - \max\left\{m, 0 \right\}} (x)_{\min\left\{n, 0 \right\}}}{(a b q x / c)_{\max\left\{k + l - m + n, 0 \right\} - 1}} 
P \c{k}{l}{m}{n}{a}{b}{c}{x}, \allowdisplaybreaks \\
R &= R (k, l, m, n)
= - \frac{x^{- \max\left\{m - 1, 0 \right\}} (x)_{\min\left\{n, 0 \right\}}}{(a b q x / c)_{\max\left\{k + l - m + n - 1, 0 \right\}}} 
P \c{k - 1}{l - 1}{m - 1}{n}{a q}{b q}{c q}{x}. 
\end{align*}
Here, $P$ is the polynomial in $x$ defined by 
\begin{align*}
&P \c{k}{l}{m}{n}{a}{b}{c}{x} \\
&:= 
\begin{cases}
\displaystyle\sum_{j = 0}^{d} \sum_{i = 0}^{\max\left\{n, 0 \right\}} \frac{\left(q^{-n} \right)_{i}}{(q)_{i}} q^{n i} 
\left(A_{j - i + m - \max\left\{m, 0 \right\}} - B_{j - i - \max\left\{m, 0 \right\}} \right) x^{j}, 
& k + l - m + n \geq 0, \\
\displaystyle\sum_{j = 0}^{d} \sum_{i = 0}^{-\min\left\{n, 0 \right\}} \frac{\left(q^{n} \right)_{i}}{(q)_{i}} 
\left(\widetilde{A}_{j - i + m - \max\left\{m, 0 \right\}} - \widetilde{B}_{j - i - \max\left\{m, 0 \right\}} \right) x^{j}, 
& k + l - m + n < 0, 
\end{cases}
\end{align*}
where $d := \max\left\{k + l - m + n, 0 \right\} + \max\left\{m, 0 \right\} - \min\left\{n, 0 \right\} - k - 1$ 
and, for any integer $j$, 
\begin{align*}
A_{j} 
&:= - \frac{(a q / c)_{k - m} (b q / c)_{l - m} (c)_{m - j - 1}}{(q^{2} / c)_{- m - 1} (q^{-j})_{j} (a)_{k - j} (b)_{l - j}} 
(c q^{m - j - 1})^{1 - n} 
\P{q^{-j}}{c q^{m - j - 1}}{a}{b}{c}{a q^{k - j}}{b q^{l - j}}{q^{1 - n}}, \allowdisplaybreaks \\
B_{j} 
&:= \frac{(a q / c)_{j} (b q / c)_{j}}{(q)_{j} (q^{2} / c)_{j}} 
\P{q^{-j}}{c q^{- j - 1}}{c q^{m - k} / a}{c q^{m - l} / b}{c q^{m}}{c q^{- j} / a}{c q^{- j} / b}{q^{k + l - m + n + 1}}, 
\allowdisplaybreaks \\
\widetilde{A}_{j} 
&:= \frac{(c)_{m} (a q / c)_{j + k - m} (b q / c)_{j + l - m}}{(a)_{k} (b)_{l} (q)_{j} (q^{2} / c)_{j - m}} 
(c q^{m - j - 1})^{-n} 
\P{q^{-j}}{c q^{m - j - 1}}{c / a}{c / b}{c}{c q^{m - k - j} / a}{c q^{m - l - j} / b}{q^{m - k - l - n + 1}}, 
\allowdisplaybreaks \\
\widetilde{B}_{j} 
&:= \frac{(a q^{-j})_{j} (b q^{-j})_{j}}{(q^{-j})_{j} (c q^{- j - 1})_{j}} q^{-j} 
\P{q^{-j}}{c q^{- j - 1}}{a q^{k}}{b q^{l}}{c q^{m}}{a q^{-j}}{b q^{-j}}{q^{1 + n}}. 
\end{align*}
\end{thm}
This theorem is a generalization of \cite[Lemma~1 and Theorem~2]{Suzuki}. 
(For an expression for $P$ as a sum of products of two ${}_{2}\phi_{1}$, see Lemma~$\ref{gen.lem:main}$, 
a generalization of \cite[Lemma~3]{Suzuki}.) 
From Theorem~$\ref{gen.main}$, we obtain the following corollary, 
a generalization of \cite[Corollary~4]{Suzuki}. 
\begin{cor}\label{cor}
The coefficients of the three-term relation $(\ref{gen.3tr})$ satisfy 
\begin{align*}
Q (k - 1, l - 1, m - 1, n) \bigg| _{(a, b, c) \mapsto (a q, b q, c q)} 
= \frac{(1 - a q) (1 - b q) x (c - a b q x)}{(1 - c) (1 - c q)} R (k, l, m, n). 
\end{align*}
\end{cor}

The following proposition provides an alternative expression for $P$. 
This proposition is a generalization of \cite[Proposition~5]{Suzuki}. 
\begin{prop}\label{gen.main2}
For any integers $k, l, m$, and $n$, with $k \leq l$, 
the polynomial $P$ defined in Theorem~$\ref{gen.main}$ can be rewritten as follows: 
\begin{align*}
&P \c{k}{l}{m}{n}{a}{b}{c}{x} \\
&= \mu 
\begin{cases}
\displaystyle\sum_{j = 0}^{d} \sum_{i = 0}^{\max\left\{n, 0 \right\}} \frac{\left(q^{-n} \right)_{i}}{(q)_{i}} q^{i} 
\left(C_{j - i} - D_{j - i + k - l} \right) x^{d - j}, 
& k + l - m + n \geq 0, \\
\displaystyle\sum_{j = 0}^{d} \sum_{i = 0}^{-\min\left\{n, 0 \right\}} \frac{\left(q^{n} \right)_{i}}{(q)_{i}} q^{(1 - n) i} 
\left(\widetilde{C}_{j - i} - \widetilde{D}_{j - i + k - l} \right) x^{d - j}, 
& k + l - m + n < 0, 
\end{cases} \nonumber 
\end{align*}
where $d := \max\left\{k + l - m + n, 0 \right\} + \max\left\{m, 0 \right\} - \min\left\{n, 0 \right\} - k - 1$ 
and, for any integer $j$, 
\begin{align*}
C_{j} 
&:= \mu_{1} 
\frac{(b)_{j} (b q / c)_{j}}{(q)_{j} (b q / a)_{j}} \left(\frac{c q}{a b} \right)^{j} 
\P{q^{-j}}{a q^{-j} / b}{q^{1 - l} / b}{c q^{m - l} / b}{a q^{k - l + 1} / b}{q^{1 - j} / b}{c q^{-j} / b}{q^{1 - n}}, 
\allowdisplaybreaks \\
D_{j} 
&:= \mu_{2} 
\frac{(q^{1 - k} / a)_{j} (c q^{m - k} / a)_{j}}{(q)_{j} (b q^{l - k + 1} / a)_{j}} q^{(1 - n) j} 
\P{q^{-j}}{a q^{k - l - j} / b}{a}{a q / c}{a q / b}{a q^{k - j}}{a q^{k - m - j + 1} / c}{q^{k + l - m + n + 1}}, 
\allowdisplaybreaks \\
\widetilde{C}_{j} 
&:= \mu_{1} 
\frac{(a q^{k})_{j} (a q^{k - m + 1} / c)_{j}}{(q)_{j} (a q^{k - l + 1} / b)_{j}} 
\left(\frac{c q^{m - k - l - n + 1}}{a b} \right)^{j} 
\P{q^{-j}}{b q^{l - k - j} / a}{q / a}{c / a}{b q / a}{q^{1 - k - j} / a}{c q^{m - k - j} / a}{q^{1 + n}}, 
\allowdisplaybreaks \\
\widetilde{D}_{j} 
&:= \mu_{2} 
\frac{(q / b)_{j} (c / b)_{j}}{(q)_{j} (a q / b)_{j}} q^{j} 
\P{q^{-j}}{b q^{-j} / a}{b q^{l}}{b q^{l - m + 1} / c}{b q^{l - k + 1} / a}{b q^{-j}}{b q^{1 - j} / c}{q^{m - k - l - n + 1}} 
\end{align*}
with 
\begin{align*}
\mu_{1} &:= q^{k (m - k - l - n + 1)} a^{m - k - l - n} \left(\frac{c}{a b} \right)^{k} 
\frac{(a)_{k} (a q / c)_{k - m}}{(a q / b)_{k - l}}, \allowdisplaybreaks \\
\mu_{2} &:= q^{l (m - k - l - n + 1)} b^{m - k - l - n} \left(\frac{c}{a b} \right)^{l} 
\frac{(b)_{l} (b q / c)_{l - m}}{(b q / a)_{l - k}}. 
\end{align*}
Here, $\mu$ is defined by 
\begin{align*}
\mu := (-1)^{k + l - m + M - N - 1} 
q^{\left\{k (k - 1) + l (l - 1) - m (m - 1) + M \left(M - 1 \right) - N \left(N - 1 \right) \right\} / 2}\, 
\frac{(q - c) a^{k} b^{l}}{(b - a) c^{m}} \left(\frac{a b}{c} \right)^{M} \frac{(c)_{m}}{(a)_{k} (b)_{l}} 
\end{align*}
with $M := \max\left\{k + l - m + n, 0 \right\}$ and $N := \min\left\{n, 0 \right\}$. 
\end{prop}

Note that the number of terms with degree $j$ in $P$ 
appearing in Theorem~$\ref{gen.main}$ increases as $j$ grows larger, 
whereas the number of terms with degree $j$ in $P$ 
appearing in Proposition~$\ref{gen.main2}$ decreases as $j$ grows larger, 
because, in general, for any non-negative integer $j$, 
${}_{4}\phi_{3} (q^{-j}, b_{1}, b_{2}, b_{3}; c_{1}, c_{2}, c_{3}; q, x)$ 
is a sum of $j + 1$ terms. 
This implies, as stated in \cite[Section~1.2]{Suzuki}, that 
both the expressions for $P$ given in Theorem~$\ref{gen.main}$ and Proposition~$\ref{gen.main2}$ 
are useful in the construction of an algorithm for calculating $P$.

The following lemma provides two generalizations of $(\ref{Gasper.tf})$, 
which are used to prove Theorem~$\ref{gen.main}$ and Proposition~$\ref{gen.main2}$. 
When $m \geq 0$ and $s = 0$, these formulas are reduced to $(\ref{Gasper.tf})$. 
\begin{lem}\label{Gasper.tf1,2}
Assume that $m \in \mathbb{Z}$ and $r, n_{1}, \dotsc, n_{r}, s \in \mathbb{Z}_{\geq 0}$. 
If \,$\big\lvert a^{-1} q^{m + 1 - (n_{1} + \dotsm + n_{r}) - s} \big\rvert < 1$, then 
\begin{align}\label{Gasper.tf1}
&{}_{r + 2}\phi_{r + 1} \biggl(\genfrac..{0pt}{}{a, \,b, \,c_{1} q^{n_{1}}, \dotsc, c_{r} q^{n_{r}}}{b q^{m + 1}, \,c_{1}, \dotsc, c_{r}}; q,\, a^{-1} q^{m + 1 - (n_{1} + \dotsm + n_{r}) - s}\biggr) \\
&= \frac{(q)_{\infty} (b q / a)_{\infty} (b q)_{m} (c_{1} / b)_{n_{1}} \dotsm (c_{r} / b)_{n_{r}}}{(q / a)_{\infty} (b q)_{\infty} (q)_{m} (c_{1})_{n_{1}} \dotsm (c_{r})_{n_{r}}} 
b^{n_{1} + \dotsm + n_{r} - m + s} \nonumber \\
&\quad \times {}_{r + 2}\phi_{r + 1}\biggl(\genfrac..{0pt}{}{q^{-m}, \,b, \,b q / c_{1}, \dotsc, b q / c_{r}}{b q / a, \,b q^{1 - n_{1}} / c_{1}, \dotsc, b q^{1 - n_{r}} / c_{r}}; q,\, q^{1 + s}\biggr). \nonumber 
\end{align}
If \,$\big\lvert a^{-1} q^{m + 1 - (n_{1} + \dotsm + n_{r}) + s + \min\left\{r - 2, 0 \right\} s} \big\rvert < 1$ 
and $\min\left\{n_{1}, \dotsc, n_{r}, s \right\} = s$, then 
\begin{align}\label{Gasper.tf2}
&\sum_{i = 0}^{s} \frac{(q^{-s})_{i}}{(q)_{i}} q^{s i} 
\frac{(q^{1 - i} / a)_{\infty} (q^{1 - n_{1}} / c_{1})_{n_{1} - i} \dotsm (q^{1 - n_{r}} / c_{r})_{n_{r} - i}}{(q)_{\infty} (q^{- m} / b)_{m - i + 1}} \\
&\quad \times {}_{r + 2}\phi_{r + 1} \biggl(\genfrac..{0pt}{}{a q^{i}, \,b q^{i}, \,c_{1} q^{n_{1}}, \dotsc, c_{r} q^{n_{r}}}{b q^{m + 1}, \,c_{1} q^{i}, \dotsc, c_{r} q^{i}}; q,\, a^{-1} q^{m + 1 - (n_{1} + \dotsm + n_{r}) + s + (r - 2) i}\biggr) \nonumber \allowdisplaybreaks \\
&= - \sum_{i = 0}^{s} \frac{(q^{-s})_{i}}{(q)_{i}} q^{i} 
\frac{(b q / a)_{\infty} (b q^{1 - n_{1} + i} / c_{1})_{n_{1} - i} \dotsm (b q^{1 - n_{r} + i} / c_{r})_{n_{r} - i}}{(b q^{i})_{\infty} (q^{i - m})_{m - i}} b^{1 - s} \nonumber \\
&\quad \times {}_{r + 2}\phi_{r + 1}\biggl(\genfrac..{0pt}{}{q^{i - m}, \,b q^{i}, \,b q / c_{1}, \dotsc, b q / c_{r}}{b q / a, \,b q^{1 - n_{1} + i} / c_{1}, \dotsc, b q^{1 - n_{r} + i} / c_{r}}; q,\, q^{1 - s}\biggr). 
\nonumber 
\end{align}
\end{lem}
Note that by using Cauchy's residue theorem, 
Chu extended $(\ref{Gasper.tf})$ to the transformation formula~\cite[(15)]{Chu}, 
which transforms a bilateral basic hypergeometric series ${}_{r + 2}\psi_{r + 2}$ into an ${}_{r + 2}\phi_{r + 1}$ series. 
From his formula, $(\ref{Gasper.tf1})$ can be derived in the following way: 
If we set $c = q$, the ${}_{r + 2}\psi_{r + 2}$ series in his formula reduces to an ${}_{r + 2}\phi_{r + 1}$ series. 
Then, setting $\epsilon = q^{-M}$, with $M \in \mathbb{Z}$, in the resulting formula and relabeling the parameters, 
we obtain $(\ref{Gasper.tf1})$. 
In Section~$2.1$, we present an independent derivation of $(\ref{Gasper.tf1})$. 

Also, as corollaries of Lemma~$\ref{Gasper.tf1,2}$, 
we give generalizations of Gasper's~\cite{Gasper} summation formulas for basic hypergeometric series.

\section{\bf{Transformation and summation formulas for basic hypergeometric series}}
In this section, we prove Lemma~$\ref{Gasper.tf1,2}$. 
Also, as corollaries of Lemma~$\ref{Gasper.tf1,2}$, 
we give some summation formulas for basic hypergeometric series. 
In order to prove Lemma~$\ref{Gasper.tf1,2}$, 
we use the following formula obtained by Andrews~\cite[p.~621, (4.1)]{Andrews}: 
\begin{align}\label{A}
\phi_{D}\biggl(\genfrac..{0pt}{}{a;\, b_{1}, \dotsc, b_{r}}{c}; x_{1}, \dotsc, x_{r}\biggr) 
= \frac{(a)_{\infty} (b_{1} x_{1})_{\infty} \!\dotsm (b_{r} x_{r})_{\infty}}{(c)_{\infty} (x_{1})_{\infty} \dotsm (x_{r})_{\infty}} 
{}_{r + 1}\phi_{r}\biggl(\genfrac..{0pt}{}{c/a,\, x_{1}, \dotsc, x_{r}}{b_{1} x_{1}, \dotsc, b_{r} x_{r}}; q,\, a\biggr), 
\end{align}
provided that $\lvert x_{1} \rvert, \dotsc, \lvert x_{r} \rvert, \lvert a \rvert < 1$ when 
each of the series on both sides does not terminate, 
where $\phi_{D}$ is the multiple series defined by 
\begin{align*}
\phi_{D}\biggl(\genfrac..{0pt}{}{a;\, b_{1}, \dotsc, b_{r}}{c}; x_{1}, \dotsc, x_{r}\biggr) 
:= \sum_{i_{1} = 0}^{\infty} \dotsm \sum_{i_{r} = 0}^{\infty} 
\frac{(a)_{i_{1} + \dotsm + i_{r}} (b_{1})_{i_{1}} \dotsm (b_{r})_{i_{r}}}{(c)_{i_{1} + \dotsm + i_{r}} (q)_{i_{1}} \dotsm (q)_{i_{r}}} 
x_{1}^{i_{1}} \dotsm x_{r}^{i_{r}}. 
\end{align*}
(This multiple series is a $q$-analogue of the Lauricella function $F_{D}$. 
See \cite[p.~228, (8.6.4)]{Slater} for the definition of $F_{D}$.) 
Also, we use the following summation formula obtained by Gasper~\cite[p.~198, (15)]{Gasper}: 
If $m, n_{1}, \dotsc, n_{r} \in \mathbb{Z}_{\geq 0}$ and $m > n_{1} + \dotsm + n_{r}$, then 
\begin{align}\label{Gasper.sf15}
{}_{r + 1}\phi_{r} \biggl(\genfrac..{0pt}{}{q^{-m}, \,c_{1} q^{n_{1}}, \dotsc, c_{r} q^{n_{r}}}{c_{1}, \dotsc, c_{r}}; q,\, q \biggr) = 0. 
\end{align}
In Section~$2. 1$, 
we first rewrite both sides of $(\ref{Gasper.tf1})$ by using $(\ref{A})$, and prove the equality. 
Next, 
we also rewrite both sides of $(\ref{Gasper.tf2})$ by using $(\ref{A})$, and then, using $(\ref{Gasper.sf15})$, 
we prove the equality. 
In Section~$2. 2$, considering the $m = 0$ and $m = -1$ cases of each $(\ref{Gasper.tf1})$ and $(\ref{Gasper.tf2})$, 
we obtain generalizations of Gasper's~\cite[p.~197--198, (7), (8), and (15)]{Gasper} 
summation formulas for basic hypergeometric series. 

Note that for any integers $m$ and $n$ satisfying $m > n$, we define a sum $\sum_{i = m}^{n} f_{i}$ as zero, 
where $\left\{f_{i} \right\}$ is an arbitrary sequence. 
Also, below, we frequently use the following identities without explicitly stating so: 
\begin{align*}
(a)_{i} = (a^{-1} q^{1 - i})_{i} (- a)^{i} q^{i (i - 1) / 2}, \;\; 
(a)_{i + j} = (a)_{i} (a q^{i})_{j}, \;\; 
(a)_{i - j} = \frac{(a)_{i}}{(a^{-1} q^{1 - i})_{j}} (- a^{-1})^{j} q^{\left\{j (j + 1) / 2 \right\} - i j}. 
\end{align*}

\subsection{Transformation formulas}\quad 
We prove Lemma~$\ref{Gasper.tf1,2}$. 

Let us put 
\begin{align*}
\widetilde{\phi}_{D}\biggl(\genfrac..{0pt}{}{a;\, b_{1}, \dotsc, b_{r}}{c}; x_{1}, \dotsc, x_{r}\biggr) 
&:= \frac{(c)_{\infty}}{(a)_{\infty}} \phi_{D}\biggl(\genfrac..{0pt}{}{a;\, b_{1}, \dotsc, b_{r}}{c}; x_{1}, \dotsc, x_{r}\biggr) \\
&= \sum_{i_{1} = 0}^{\infty} \dotsm \sum_{i_{r} = 0}^{\infty} 
\frac{(c q^{i_{1} + \dotsm + i_{r}})_{\infty} (b_{1})_{i_{1}} \dotsm (b_{r})_{i_{r}}}{(a q^{i_{1} + \dotsm + i_{r}})_{\infty} (q)_{i_{1}} \dotsm (q)_{i_{r}}} 
x_{1}^{i_{1}} \dotsm x_{r}^{i_{r}}. 
\end{align*}
Then, $(\ref{A})$ can be rewritten as 
\begin{align}\label{A'}
{}_{r + 1}\phi_{r}\biggl(\genfrac..{0pt}{}{a,\, b_{1}, \dotsc, b_{r}}{c_{1}, \dotsc, c_{r}}; q,\, x\biggr) 
= \frac{(b_{1})_{\infty} \dotsm (b_{r})_{\infty}}{(c_{1})_{\infty} \dotsm (c_{r})_{\infty}} 
\widetilde{\phi}_{D}\biggl(\genfrac..{0pt}{}{x;\, c_{1} / b_{1}, \dotsc, c_{r} / b_{r}}{a x}; b_{1}, \dotsc, b_{r}\biggr), 
\end{align}
provided that $\lvert x \rvert, \lvert b_{1} \rvert, \dotsc, \lvert b_{r} \rvert < 1$ when 
each of the series on both sides does not terminate. 
In order to prove Lemma~$\ref{Gasper.tf1,2}$, we use $(\ref{A'})$ and the following two lemmas and a corollary. 

Reversing the order of summation in a $\widetilde{\phi}_{D}$ multiple series, we obtain the following lemma: 
\begin{lem}\label{lem:2-1}
If $m \in \mathbb{Z},\ r, n_{1}, \dotsc, n_{r} \in \mathbb{Z}_{\geq 0}$, 
$a \notin q^{\mathbb{Z}}$, and $\lvert x \rvert < 1$, 
then 
\begin{align}\label{lem:2-1-1}
&\widetilde{\phi}_{D}\biggl(\genfrac..{0pt}{}{a;\, b,\, q^{-n_{1}}, \dotsc, q^{-n_{r}}}{q^{m+1}}; x,\, x_{1}, \dotsc, x_{r}\biggr) \\
&= (-1)^{n_{1} + \dotsm + n_{r}} 
q^{- \left\{n_{1} (n_{1} + 1) + \dotsm + n_{r} (n_{r} + 1) \right\} / 2} 
x^{-m - (n_{1} + \dotsm + n_{r})} x_{1}^{n_{1}} \dotsm x_{r}^{n_{r}} \nonumber \\
&\; \times 
\sum_{i_{1} = 0}^{n_{1}} \!\dotsm\! \sum_{i_{r} = 0}^{n_{r}} 
\sum_{j = \max\left\{0,\, I \right\}}^{\infty} 
\frac{(b)_{j - I} (q^{1 + j - I})_{\infty} (q^{-n_{1}})_{i_{1}} \dotsm (q^{-n_{r}})_{i_{r}}}{(a q^{- m + j})_{\infty} (q)_{j} (q)_{i_{1}} \dotsm (q)_{i_{r}}} x^{j} \left(\frac{x q^{n_{1} + 1}}{x_{1}} \right)^{i_{1}} \!\dotsm \left(\frac{x q^{n_{r} + 1}}{x_{r}} \right)^{i_{r}}, 
\nonumber 
\end{align}
where $I := m + (n_{1} + \dotsm + n_{r}) - (i_{1} + \dotsm + i_{r})$. 
\end{lem}
\begin{proof}
From the definition of $\widetilde{\phi_{D}}$, the left-hand side of $(\ref{lem:2-1-1})$ is equal to 
\begin{align}\label{lem:2-1-2}
\sum_{i_{1} = 0}^{n_{1}} \dotsm \sum_{i_{r} = 0}^{n_{r}} 
\sum_{j = 0}^{\infty} 
\frac{(q^{m + 1 + (i_{1} + \dotsm + i_{r}) + j})_{\infty} (b)_{j} (q^{-n_{1}})_{i_{1}} \dotsm (q^{-n_{r}})_{i_{r}}}{(a q^{(i_{1} + \dotsm + i_{r}) + j})_{\infty} (q)_{j} (q)_{i_{1}} \dotsm (q)_{i_{r}}} x^{j} x_{1}^{i_{1}} \dotsm x_{r}^{i_{r}}. 
\end{align}
Replacing $i_{\nu} \,(\nu = 1, \dotsc, r)$ by $n_{\nu} - i_{\nu}$, respectively, 
we find that $(\ref{lem:2-1-2})$ can be rewritten as 
\begin{align}\label{lem:2-1-3}
&(-1)^{n_{1} + \dotsm + n_{r}} 
q^{- \left\{n_{1} (n_{1} + 1) + \dotsm + n_{r} (n_{r} + 1) \right\} / 2} 
x_{1}^{n_{1}} \dotsm x_{r}^{n_{r}} \\
&\quad \times 
\sum_{i_{1} = 0}^{n_{1}} \dotsm \sum_{i_{r} = 0}^{n_{r}} 
\sum_{j = 0}^{\infty} 
\frac{(q^{1 + I + j})_{\infty} (b)_{j} (q^{-n_{1}})_{i_{1}} \dotsm (q^{-n_{r}})_{i_{r}}}{(a q^{-m + I + j})_{\infty} (q)_{j} (q)_{i_{1}} \dotsm (q)_{i_{r}}} x^{j} \left(\frac{q^{n_{1} + 1}}{x_{1}} \right)^{i_{1}} \dotsm \left(\frac{q^{n_{r} + 1}}{x_{r}} \right)^{i_{r}}. 
\nonumber 
\end{align}
Moreover, replacing $j$ by $j - I$, we find that $(\ref{lem:2-1-3})$ is equal to the right-hand side of $(\ref{lem:2-1-1})$, 
because $1 / (q)_{j} = 0$ holds for any negative integer $j$. 
Thus the lemma is proved. 
\end{proof}

From Lemma~$\ref{lem:2-1}$, we obtain the following corollary: 
\begin{cor}\label{cor:2-2}
If $m \in \mathbb{Z},\ r, n_{1}, \dotsc, n_{r} \in \mathbb{Z}_{\geq 0}$, 
$a \notin q^{\mathbb{Z}},\ \left(b q^{- m - (n_{1} + \dotsm + n_{r})} \right)_{m + (n_{1} + \dotsm + n_{r})} \neq 0$, 
and $\lvert x \rvert < 1$, 
then 
\begin{align}\label{cor:2-2-1}
&\widetilde{\phi}_{D}\biggl(\genfrac..{0pt}{}{a;\, b,\, q^{-n_{1}}, \dotsc, q^{-n_{r}}}{q^{m+1}}; x,\, x_{1}, \dotsc, x_{r}\biggr) \\
&= \frac{(b)_{\infty}}{(a q^{-m})_{\infty}} (-1)^{n_{1} + \dotsm + n_{r}} 
q^{- \left\{n_{1} (n_{1} + 1) + \dotsm + n_{r} (n_{r} + 1) \right\} / 2} 
x^{- m - (n_{1} + \dotsm + n_{r})} x_{1}^{n_{1}} \dotsm x_{r}^{n_{r}} \nonumber \\ 
&\quad \times 
\widetilde{\phi}_{D}\biggl(\genfrac..{0pt}{}{b q^{- m - (n_{1} + \dotsm + n_{r})};\, a q^{-m},\, q^{-n_{1}}, \dotsc, q^{-n_{r}}}{q^{1 - m - (n_{1} + \dotsm + n_{r})}}; x,\, \frac{x q^{n_{1} + 1}}{x_{1}}, \dotsc, \frac{x q^{n_{r} + 1}}{x_{r}}\biggr). \nonumber 
\end{align}
\end{cor}
\begin{proof}
If $\left(b q^{- m - (n_{1} + \dotsm + n_{r})} \right)_{m + (n_{1} + \dotsm + n_{r})} \neq 0$, then, 
eliminating terms that are zero on the right-hand side of $(\ref{cor:2-2-1})$, 
we are able to reduce $(\ref{cor:2-2-1})$ to $(\ref{lem:2-1-1})$. 
This proves the corollary. 
\end{proof}

From $(\ref{Gasper.sf15})$, we obtain the following lemma: 
\begin{lem}\label{lem:2-3}
If $r, n_{1}, \dotsc, n_{r}, s \in \mathbb{Z}_{\geq 0}$, 
then 
\begin{align*}
\sum_{i = 0}^{s} \frac{(q^{-s})_{i}}{(q)_{i}} q^{i} 
\sum_{i_{1} = 0}^{n_{1}} \dotsm \sum_{i_{r} = 0}^{n_{r}} 
\sum_{j = 0}^{s - 1 - I_{r}} 
\frac{(q^{1 - s})_{I_{r} + j} (a q^{-i})_{j} (b_{1} q^{i})_{i_{1}} \dotsm (b_{r} q^{i})_{i_{r}}}{(c q^{i})_{1 - s + I_{r} + j} (q)_{j} (q)_{i_{1}} \dotsm (q)_{i_{r}}} (x q^{i})^{j} x_{1}^{i_{1}} \dotsm x_{r}^{i_{r}} 
= 0, 
\end{align*}
where $I_{r} := i_{1} + \dotsm + i_{r}$. 
\end{lem}
\begin{proof}
Changing the order of summation, we rewrite the left-hand side as 
\begin{align*}
&\sum_{i_{1} = 0}^{n_{1}} \dotsm \sum_{i_{r} = 0}^{n_{r}} 
\sum_{j = 0}^{s - 1 - I_{r}} 
\frac{(q^{1 - s})_{I_{r} + j} (a)_{j} (b_{1})_{i_{1}} \dotsm (b_{r})_{i_{r}}}{(c)_{1 - s + I_{r} + j} (q)_{j} (q)_{i_{1}} \dotsm (q)_{i_{r}}} 
x^{j} x_{1}^{i_{1}} \dotsm x_{r}^{i_{r}} \\
&\quad \times 
{}_{r + 3}\phi_{r + 2}\biggl(\genfrac..{0pt}{}{q^{-s},\, c,\, q / a,\, b_{1} q^{i_{1}}, \dotsc, b_{r} q^{i_{r}}}{c q^{1 - s + I_{r} + j},\, q^{1 - j} / a,\, b_{1}, \dotsc, b_{r}}; q,\, q\biggr). 
\end{align*}
Then, from $(\ref{Gasper.sf15})$, we find that the sum of the above ${}_{r + 3}\phi_{r + 2}$ series equals zero. 
This proves the lemma. 
\end{proof}

\begin{remark}
It follows from Lemma~$\ref{lem:2-3}$ that 
if $r, s \in \mathbb{Z}_{\geq 0}$ 
and $c \notin \left\{q^{-m} \mid m \in \mathbb{Z}_{\geq 0} \right\}$, then 
\begin{align*}
\sum_{i = 0}^{s} \frac{(q^{-s})_{i}\, q^{i}}{(q)_{i} (c q^{i})_{1 - s}} 
\phi_{D}\biggl(\genfrac..{0pt}{}{q^{1 - s};\, a q^{-i}, b_{1} q^{i}, \dotsc, b_{r} q^{i}}{c q^{i + 1 - s}}; x q^{i}, x_{1}, \dotsc, x_{r}\biggr) 
= 0. 
\end{align*}
\end{remark}


Now, let us prove Lemma~$\ref{Gasper.tf1,2}$. 
It is sufficient to prove this lemma under 
the additional assumption that $\lvert b \rvert < 1$, 
because, from the uniqueness of analytic continuation, 
we may drop this assumption. 
Thus, below, we impose the additional assumption. 

First, we prove $(\ref{Gasper.tf1})$ by using $(\ref{A'})$ and Corollary~$\ref{cor:2-2}$. 
From $(\ref{A'})$, the left- and the right-hand sides of $(\ref{Gasper.tf1})$ can be rewritten as 
\begin{align*}
&\frac{(b)_{m+1}}{(c_{1})_{n_{1}} \dotsm (c_{r})_{n_{r}}} 
\widetilde{\phi}_{D} \biggl(\genfrac..{0pt}{}{a^{-1} q^{m + 1 - (n_{1} + \dotsm + n_{r}) - s};\, q^{m+1},\, q^{-n_{1}}, \dotsc, q^{-n_{r}}}{q^{m + 1 - (n_{1} + \dotsm + n_{r}) - s}}; b,\, c_{1} q^{n_{1}}, \dotsc, c_{r} q^{n_{r}}\biggr), 
\allowdisplaybreaks \\
&\frac{(b)_{m+1}}{(c_{1})_{n_{1}} \dotsm (c_{r})_{n_{r}}} \frac{(q^{m+1})_{\infty}}{(q / a)_{\infty}} 
(-1)^{n_{1} + \dotsm + n_{r}} q^{\left\{n_{1} (n_{1} - 1) + \dotsm + n_{r} (n_{r} - 1) \right\} / 2} 
b^{- m + s} c_{1}^{n_{1}} \dotsm c_{r}^{n_{r}} \\
&\quad \times 
\widetilde{\phi}_{D} \biggl(\genfrac..{0pt}{}{q^{1 + s};\, q / a,\, q^{-n_{1}}, \dotsc, q^{-n_{r}}}{q^{1 - m + s}}; b,\, \frac{b q}{c_{1}}, \dotsc, \frac{b q}{c_{r}}\biggr), 
\end{align*}
respectively. Then, from Corollary~$\ref{cor:2-2}$, we find that these two expressions are equal to each other. 
This completes the proof of $(\ref{Gasper.tf1})$.

Next, we prove $(\ref{Gasper.tf2})$ 
by using $(\ref{A'})$ and Lemmas~$\ref{lem:2-1}$ and $\ref{lem:2-3}$. 
Note that $1 / (q^{i - m})_{m - i} = 0$ holds for $i > m$. 
From this and $(\ref{A'})$, 
the right-hand side of $(\ref{Gasper.tf2})$ can be rewritten as 
\begin{align}\label{G2-right}
&-\sum_{i = 0}^{\min\left\{m, s \right\}} \frac{(q^{-s})_{i}}{(q)_{i}} q^{i} \frac{b^{1 - s}}{(q^{i - m})_{m - i}} 
\widetilde{\phi}_{D} \biggl(\genfrac..{0pt}{}{q^{1 - s};\, q^{1 - i} / a,\, q^{i - n_{1}}, \dotsc, q^{i - n_{r}}}{q^{i - m + 1 - s}}; b q^{i},\, \frac{b q}{c_{1}}, \dotsc, \frac{b q}{c_{r}}\biggr). 
\end{align}
On the other hand, from $(\ref{A'})$, the left-hand side of $(\ref{Gasper.tf2})$ can be rewritten as 
\begin{align}\label{G2-left}
&\sum_{i = 0}^{s} \frac{(q^{-s})_{i}}{(q)_{i}} q^{s i} 
\frac{(q^{1 - i} / a)_{\infty} (q^{1 - n_{1}} / c_{1})_{n_{1} - i} \dotsm (q^{1 - n_{r}} / c_{r})_{n_{r} - i}}{(q)_{\infty} (q^{- m} / b)_{m - i + 1}} 
\frac{(b q^{i})_{m + 1 - i}}{(c_{1} q^{i})_{n_{1} - i} \dotsm (c_{r} q^{i})_{n_{r} - i}} \\
&\quad \times 
\widetilde{\phi}_{D} \biggl(\genfrac..{0pt}{}{a^{-1} q^{m + 1 - (n_{1} + \dotsm + n_{r}) + s + (r - 2)i};\, q^{m + 1 - i},\, q^{i - n_{1}}, \dotsc, q^{i - n_{r}}}{q^{m + 1 - (n_{1} + \dotsm + n_{r}) + s + (r - 1)i}}; b q^{i},\, c_{1} q^{n_{1}}, \dotsc, c_{r} q^{n_{r}}\biggr). \nonumber 
\end{align}
Below, we show that $(\ref{G2-left})$ is equal to $(\ref{G2-right})$. 
It follows from Lemma~$\ref{lem:2-1}$ that $(\ref{G2-left})$ equals 
\begin{align}\label{G2-left2}
&\sum_{i = 0}^{s} \frac{(q^{-s})_{i}}{(q)_{i}} (-1)^{m + 1 - i} q^{\left\{(m - i) (m - i - 1) / 2\right\} + m} b^{1 - s} \\
&\times \sum_{i_{1} = 0}^{n_{1} - i} \!\dotsm\! \sum_{i_{r} = 0}^{n_{r} - i} 
\sum_{j = \max\left\{0,\, I'\right\}}^{\infty} 
\!\!\!\frac{(q^{m + 1 - i})_{j - I'} (q^{1 - i} / a)_{j} (q^{i - n_{1}})_{i_{1}} \dotsm (q^{i - n_{r}})_{i_{r}}}{(q)_{j - I'} (q)_{j} (q)_{i_{1}} \dotsm (q)_{i_{r}}} (b q^{i})^{j} \left(\frac{b q}{c_{1}} \right)^{i_{1}} \!\!\dotsm \left(\frac{b q}{c_{r}} \right)^{i_{r}}\!\!, \nonumber 
\end{align}
where $I' := m - i + s - I_{r}$ with $I_{r} := i_{1} + \dotsm + i_{r}$. 
Since 
\begin{align*}
\frac{(q^{m + 1 - i})_{j - I'}}{(q)_{j - I'}} 
= (-1)^{m + 1 - i + \delta(i \leq m)} q^{-\left\{(m - i) (m - i - 1) / 2\right\} - m + i} 
\frac{(q^{1 - s})_{I_{r} + j}}{(q^{1 - s})_{s - 1} (q^{i - m})_{1 - s + I_{r} + j}}, 
\end{align*}
where $\delta(i \leq m) := 1\, (i \leq m)$ and $0\, (i > m)$, 
putting 
\begin{align*}
H (i, i_{1}, \dotsc, i_{r}, j) 
:= \frac{(q^{1 - s})_{I_{r} + j} (q^{1 - i} / a)_{j} (q^{i - n_{1}})_{i_{1}} \dotsm (q^{i - n_{r}})_{i_{r}}}{(q^{i - m})_{1 - s + I_{r} + j} (q)_{j} (q)_{i_{1}} \dotsm (q)_{i_{r}}} 
(b q^{i})^{j} \left(\frac{b q}{c_{1}} \right)^{i_{1}} \dotsm \left(\frac{b q}{c_{r}} \right)^{i_{r}}, 
\end{align*}
we rewrite $(\ref{G2-left2})$ as 
\begin{align}\label{G2-left3}
\sum_{i = 0}^{s} \frac{(q^{-s})_{i}}{(q)_{i}} q^{i} (-1)^{\delta(i \leq m)} 
\frac{b^{1 - s}}{(q^{1 - s})_{s - 1}} 
\sum_{i_{1} = 0}^{n_{1} - i} \dotsm \sum_{i_{r} = 0}^{n_{r} - i} 
\sum_{j = \max\left\{0,\, I'\right\}}^{\infty} H (i, i_{1}, \dotsc, i_{r}, j). 
\end{align}
Moreover, considering cases satisfying $H (i, i_{1}, \dotsc, i_{r}, j) = 0$, we find that $(\ref{G2-left3})$ equals 
\begin{align}\label{G2-left4}
&- \frac{b^{1 - s}}{(q^{1 - s})_{s - 1}} 
\sum_{i = 0}^{\min\left\{m, s \right\}} \frac{(q^{-s})_{i}}{(q)_{i}} q^{i} 
\sum_{i_{1} = 0}^{n_{1} - i} \dotsm \sum_{i_{r} = 0}^{n_{r} - i} 
\sum_{j = \max\left\{0,\, s - I_{r}\right\}}^{\infty} H (i, i_{1}, \dotsc, i_{r}, j) \\
&\quad + \frac{b^{1 - s}}{(q^{1 - s})_{s - 1}} 
\sum_{i = \max\left\{0, m + 1 \right\}}^{s} \frac{(q^{-s})_{i}}{(q)_{i}} q^{i} 
\sum_{i_{1} = 0}^{n_{1} - i} \dotsm \sum_{i_{r} = 0}^{n_{r} - i} 
\sum_{j = 0}^{s - 1 - I_{r}} H (i, i_{1}, \dotsc, i_{r}, j). \nonumber 
\end{align}
It follows from Lemma~$\ref{lem:2-3}$ that the second multiple series in $(\ref{G2-left4})$ equals 
\begin{align*}
-\frac{b^{1 - s}}{(q^{1 - s})_{s - 1}} 
\sum_{i = 0}^{\min\left\{m, s \right\}} \frac{(q^{-s})_{i}}{(q)_{i}} q^{i} 
\sum_{i_{1} = 0}^{n_{1} - i} \dotsm \sum_{i_{r} = 0}^{n_{r} - i} 
\sum_{j = 0}^{s - 1 - I_{r}} H (i, i_{1}, \dotsc, i_{r}, j). 
\end{align*}
This implies that $(\ref{G2-left4})$ equals 
\begin{align}\label{G2-left5}
-\frac{b^{1 - s}}{(q^{1 - s})_{s - 1}} 
\sum_{i = 0}^{\min\left\{m, s \right\}} \frac{(q^{-s})_{i}}{(q)_{i}} q^{i} 
\sum_{i_{1} = 0}^{n_{1} - i} \dotsm \sum_{i_{r} = 0}^{n_{r} - i} 
\sum_{j = 0}^{\infty} H (i, i_{1}, \dotsc, i_{r}, j). 
\end{align}
It is easily verified that $(\ref{G2-left5})$ is equal to $(\ref{G2-right})$. 
This completes the proof of $(\ref{Gasper.tf2})$.

\subsection{Summation formulas}\quad 
By using Lemma~$\ref{Gasper.tf1,2}$, we derive some summation formulas. 
Below, it is assumed that $r, n_{1}, \dotsc, n_{r}, s \in \mathbb{Z}_{\geq 0}$. 

By setting $m = 0$ in $(\ref{Gasper.tf1})$, we find that 
if $\big\lvert a^{-1} q^{1 - (n_{1} + \dotsm + n_{r}) - s} \big\rvert < 1$, then 
\begin{align}\label{Gasper.sf1}
&{}_{r + 2}\phi_{r + 1} \biggl(\genfrac..{0pt}{}{a, \,b, \,c_{1} q^{n_{1}}, \dotsc, c_{r} q^{n_{r}}}{b q, \,c_{1}, \dotsc, c_{r}}; q,\, a^{-1} q^{1 - (n_{1} + \dotsm + n_{r}) - s}\biggr) \\
&= \frac{(q)_{\infty} (b q / a)_{\infty} (c_{1} / b)_{n_{1}} \dotsm (c_{r} / b)_{n_{r}}}{(q / a)_{\infty} (b q)_{\infty} (c_{1})_{n_{1}} \dotsm (c_{r})_{n_{r}}} 
b^{n_{1} + \dotsm + n_{r} + s}. \nonumber 
\end{align}
(This follows also from the $c = q$ case of Chu's~\cite[(14)]{Chu} bilateral summation formula.) 
By setting $m = 0$ in $(\ref{Gasper.tf2})$, we find that 
if $\big\lvert a^{-1} q^{1 - (n_{1} + \dotsm + n_{r}) + s + \min\left\{r - 2, 0 \right\} s} \big\rvert < 1$ 
and $\min\left\{n_{1}, \dotsc, n_{r}, s \right\} = s$, then 
\begin{align}\label{Gasper.sf2}
&\sum_{i = 0}^{s} \frac{(q^{-s})_{i}}{(q)_{i}} q^{s i} 
\frac{(q^{1 - i} / a)_{\infty} (q^{1 - n_{1}} / c_{1})_{n_{1} - i} \dotsm (q^{1 - n_{r}} / c_{r})_{n_{r} - i}}{(q)_{\infty} (1 / b)_{1 - i}} \\
&\quad \times {}_{r + 2}\phi_{r + 1} \biggl(\genfrac..{0pt}{}{a q^{i}, \,b q^{i}, \,c_{1} q^{n_{1}}, \dotsc, c_{r} q^{n_{r}}}{b q, \,c_{1} q^{i}, \dotsc, c_{r} q^{i}}; q,\, a^{-1} q^{1 - (n_{1} + \dotsm + n_{r}) + s + (r - 2) i}\biggr) \nonumber \allowdisplaybreaks \\
&= - \frac{(b q / a)_{\infty} (b q^{1 - n_{1}} / c_{1})_{n_{1}} \dotsm (b q^{1 - n_{r}} / c_{r})_{n_{r}}}{(b)_{\infty}} b^{1 - s}. 
\nonumber 
\end{align}
When $s = 0$, both $(\ref{Gasper.sf1})$ and $(\ref{Gasper.sf2})$ reduce to 
the formula~\cite[p.~197, (7)]{Gasper}.

It follows from the $m = - 1$ case of $(\ref{Gasper.tf1})$ that 
if $\big\lvert a^{-1} q^{- (n_{1} + \dotsm + n_{r}) - s} \big\rvert < 1$, then 
\begin{align}\label{Gasper.sf3}
{}_{r + 1}\phi_{r} \biggl(\genfrac..{0pt}{}{a, \,c_{1} q^{n_{1}}, \dotsc, c_{r} q^{n_{r}}}{c_{1}, \dotsc, c_{r}}; q,\, a^{-1} q^{- (n_{1} + \dotsm + n_{r}) - s}\biggr) 
= 0. 
\end{align}
Substituting $a = q^{- (n_{1} + \dotsm + n_{r}) - s - 1}$ into $(\ref{Gasper.sf3})$ gives 
the formula~\cite[p.~198, (15)]{Gasper}, 
which is introduced in $(\ref{Gasper.sf15})$ and used to prove $(\ref{Gasper.tf2})$. 
It follows from the $m = - 1$ case of $(\ref{Gasper.tf2})$ that 
if $\big\lvert a^{-1} q^{- (n_{1} + \dotsm + n_{r}) + s + \min\left\{r - 2, 0 \right\} s} \big\rvert < 1$ and 
$\min\left\{n_{1}, \dotsc, n_{r}, s \right\} = s$, then 
\begin{align}\label{Gasper.sf4}
&\sum_{i = 0}^{s} \frac{(q^{-s})_{i}}{(q)_{i}} q^{s i} 
\frac{(q^{1 - i} / a)_{\infty} (q^{1 - n_{1}} / c_{1})_{n_{1} - i} \dotsm (q^{1 - n_{r}} / c_{r})_{n_{r} - i}}{(q / b)_{- i}} \\
&\quad \times {}_{r + 2}\phi_{r + 1} \biggl(\genfrac..{0pt}{}{a q^{i}, \,b q^{i}, \,c_{1} q^{n_{1}}, \dotsc, c_{r} q^{n_{r}}}{b, \,c_{1} q^{i}, \dotsc, c_{r} q^{i}}; q,\, a^{-1} q^{- (n_{1} + \dotsm + n_{r}) + s + (r - 2) i}\biggr) = 0. \nonumber 
\end{align}
When $s = 0$, both $(\ref{Gasper.sf3})$ and $(\ref{Gasper.sf4})$ reduce to 
the formula~\cite[p.~197, (8)]{Gasper}.

Of course, considering limit cases and inversions of $(\ref{Gasper.sf1})\mbox{--}(\ref{Gasper.sf4})$ 
as in \cite{Gasper}, 
we are able to derive generalizations of Gasper's~\cite{Gasper} other summation formulas. 

\section{\bf{Three-term relations for the basic hypergeometric series ${}_{2}\phi_{1}$}}
In this section, following the notation, the method, and the process described in \cite{Suzuki}, 
we prove Theorem~$\ref{gen.main}$, Corollary~$\ref{cor}$, and Proposition~$\ref{gen.main2}$. 
In Sections~$3. 1$--$3. 6$, we prove Theorem~$\ref{gen.main}$ in the following way: 
In Section~$3. 1$, following the method described in \cite[Section~$2$]{Suzuki} 
(that is, employing the method due to Vid\=unas \cite[Section~$3$]{Vidunas}), 
we prove the uniqueness of the pair $(Q, R)$ satisfying $(\ref{gen.3tr})$. 
In Section~$3. 2$, we introduce a series ${}_{2}\widetilde{\phi}_{1}$ 
and rewrite the three-term relation $(\ref{gen.3tr})$ 
into a three-term relation for ${}_{2}\widetilde{\phi}_{1}$ with coefficients $\widetilde{Q}$ and $\widetilde{R}$. 
Then, comparing these two three-term relations, 
we obtain expressions for $Q$ and $R$ in terms of $\widetilde{Q}$ and $\widetilde{R}$, respectively. 
In Section~$3. 3$, we introduce four local solutions $y_{i}$ ($i = 1, 2, 3, 4$) 
of a $q$-differential equation $E_{q} (a, b, c; x)$ 
defined by $L y (x) = 0$, 
where 
\begin{align*}
L := (1 - T_{q}) (1 - c q^{-1} T_{q}) - x (1 - a T_{q}) (1 - b T_{q}) 
\end{align*}
with $T_{q} : x \mapsto q x$. 
Also, in Section~$3. 4$, we introduce eight $q$-differential operators called ``contiguity operators,'' 
and combining these operators, 
we obtain a $q$-differential operator $\theta (k, l, m; n)$. 
Then, in Section~$3. 5$, operating $\theta (k, l, m; n)$ on $y_{i}$'s 
and using the uniqueness proved in Section~$3. 1$, 
we obtain linear equations for $\widetilde{Q}$ and $\widetilde{R}$. 
Solving these equations, we are able to express each $\widetilde{Q}$ and $\widetilde{R}$ 
as a ratio of infinite series defined as a sum of products of $y_{i}$'s. 
In Section~$3. 6$, 
using these expressions and Lemma~$\ref{Gasper.tf1,2}$, 
we obtain expressions for $\widetilde{Q}$ and $\widetilde{R}$ in terms of $P$ 
and thereby complete the proof of Theorem~$\ref{gen.main}$. 
In Section~$3. 7$, to prove Proposition~$\ref{gen.main2}$, 
we first define a polynomial $\widetilde{P}$ and obtain an expression for $Q$ in terms of $\widetilde{P}$ 
that differs from the expression in terms of $P$ given in Theorem~$\ref{gen.main}$. 
Then, employing the uniqueness of $Q$, by comparing these expressions for $Q$, 
we are able to complete the proof. 
In Section~$3. 8$, by using Theorem~$\ref{gen.main}$, we prove Corollary~$\ref{cor}$. 


\subsection{Uniqueness of the pair $(Q, R)$}\quad 
We prove by contradiction that the pair $(Q, R)$ 
satisfying $(\ref{gen.3tr})$ is uniquely determined by $(k, l, m, n)$. 

Let us assume that there are two distinct pairs $(Q_{1}, R_{1})$ and $(Q_{2}, R_{2})$ of rational functions 
satisfying $(\ref{gen.3tr})$. Then, we have 
\begin{align*}
(Q_{1} - Q_{2}) \cdot \p{a q}{b q}{c q}{x} = (R_{2} - R_{1}) \cdot \p{a}{b}{c}{x}. 
\end{align*}
Therefore, it turns out that ${}_{2}\phi_{1} (a q, b q; c q; q, x) \big/ {}_{2}\phi_{1} (a, b; c; q, x)$ 
is a rational function of $a, b, c, q$, and $x$. 
However, this leads to a contradiction 
(see the last eight lines of \cite[Section~2.1]{Suzuki} for detail). 
Thus the uniqueness of the pair $(Q, R)$ satisfying $(\ref{gen.3tr})$ is proved.

\subsection{Rewriting of the three-term relation}\quad 
We introduce a series ${}_{2}\widetilde{\phi}_{1}$ and rewrite the three-term relation $(\ref{gen.3tr})$ 
into a three-term relation for ${}_{2}\widetilde{\phi}_{1}$. 

Let ${}_{2}\widetilde{\phi}_{1}$ be the series defined by 
\begin{align*}
\f{a}{b}{c}{x} := \frac{(q)_{\infty} (c)_{\infty}}{(a)_{\infty} (b)_{\infty}} \p{a}{b}{c}{x}. 
\end{align*}
Then, the three-term relation $(\ref{gen.3tr})$ can be rewritten as 
\begin{align}\label{gen.3tr'}
\f{a q^{k}}{b q^{l}}{c q^{m}}{x q^{n}} 
= \widetilde{Q} \cdot \f{a q}{b q}{c q}{x} + \widetilde{R} \cdot \f{a}{b}{c}{x}. 
\end{align}
Comparing $(\ref{gen.3tr})$ with $(\ref{gen.3tr'})$, 
we find that $Q$ and $R$ can be expressed in terms of $\widetilde{Q}$ and $\widetilde{R}$ as 
\begin{align}
Q &= \frac{(c q)_{m - 1}}{(a q)_{k - 1} (b q)_{l - 1}} \widetilde{Q}, \label{gen.Q,tQ} \\
R &= \frac{(c)_{m}}{(a)_{k} (b)_{l}} \widetilde{R}. \label{gen.R,tR}
\end{align}
Hence, we investigate $\widetilde{Q}$ and $\widetilde{R}$. 
Note that as a consequence of Section~$3. 1$, it follows that 
the pair $(\widetilde{Q}, \widetilde{R})$ satisfying $(\ref{gen.3tr'})$ is uniquely determined by $(k, l, m, n)$.

\subsection{Local solutions of $E_{q} (a, b, c; x)$}\quad 
We introduce four local solutions of $E_{q} (a, b, c; x)$. 

Let $y_{i}$ ($i = 1, 2, 3, 4$) be the functions defined by 
\begin{align*}
\y{1}{a}{b}{c}{x} = y_{1} (a, b; c; x) &:= \f{a}{b}{c}{x}, \allowdisplaybreaks \\
\y{2}{a}{b}{c}{x} = y_{2} (a, b; c; x) &:= x^{1 - \gamma} \f{a q / c}{b q / c}{q^{2} / c}{x}, \allowdisplaybreaks \\
\y{3}{a}{b}{c}{x} = y_{3} (a, b; c; x) &:= a^{\gamma - \alpha - \beta + 1} x^{-\alpha} 
\f{a}{a q / c}{a q / b}{\frac{c q}{a b x}}, \allowdisplaybreaks \\
\y{4}{a}{b}{c}{x} = y_{4} (a, b; c; x) &:= b^{\gamma - \alpha - \beta + 1} x^{-\beta} 
\f{b}{b q / c}{b q / a}{\frac{c q}{a b x}}, 
\end{align*}
where $a = q^{\alpha},\ b = q^{\beta}$, and $c = q^{\gamma}$. 

Then, using the general theory of linear difference equations mentioned 
in \cite[p.~62, Theorem~2. 15]{Elaydi}, we are able to verify the following lemma: 
\begin{lem}\label{lem:3-1}
Under the assumption $(\ref{assump})$, 
$y_{i} (a, b; c; x) \, (i = 1, 2)$ are linearly independent solutions around $x = 0$ 
of $E_{q} (a, b, c; x)$, and $y_{i} (a, b; c; x) \, (i = 3, 4)$ are linearly independent solutions around $x = \infty$. 
\end{lem}

\begin{remark}\label{rem1}
The Casoratians are given by 
\begin{align*}
\det 
\left( 
\begin{array}{rr}
y_{1} & y_{2} \\
T_{q} y_{1} & T_{q} y_{2} 
\end{array}
\right) 
&= - \frac{(q)_{\infty}^{2} (c)_{\infty} (q / c)_{\infty} (a b q x / c)_{\infty}}{(a)_{\infty} (b)_{\infty} (a q / c)_{\infty} (b q / c)_{\infty} (x)_{\infty}} x^{1 - \gamma}, \allowdisplaybreaks \\
\det 
\left( 
\begin{array}{rr}
y_{3} & y_{4} \\
T_{q} y_{3} & T_{q} y_{4} 
\end{array}
\right) 
&= \frac{a (a b)^{\gamma - \alpha - \beta} (q)_{\infty}^{2} (a q / b)_{\infty} (b / a)_{\infty} (q / x)_{\infty}}{(a)_{\infty} (b)_{\infty} (a q / c)_{\infty} (b q / c)_{\infty} (c / (a b x))_{\infty}} x^{- \alpha - \beta}. 
\end{align*}
For these expressions, see Remark~$\ref{rem2},\ (\ref{Y(1,1,1,0)})$, and $(\ref{tY(1,1,1,0)})$. 
\end{remark}

\subsection{Contiguity operators}\quad 
We introduce eight $q$-differential operators 
that each increase or decrease $a, b, c$, or $x$ by $q$ times, 
and then, combining these operators, we obtain a linear isomorphism $\theta (k, l, m; n)$, 
which sends $a, b, c$, and $x$ to $a q^{k}, b q^{l}, c q^{m}$, and $x q^{n}$, respectively. 

Let $H_{j}$ and $B_{j}$ ($j = 1, 2, 3, 4$) be the $q$-differential operators 
defined by 
\begin{align*}
H_{1} (a, b, c; x) &:= 1 - a T_{q}, \allowdisplaybreaks \\
H_{2} (a, b, c; x) &:= 1 - b T_{q}, \allowdisplaybreaks \\
H_{3} (a, b, c; x) &:= \left\{(c - a) (c - b) x \right\}^{-1} 
\left\{c^{2} + a b x - (a + b) c x - c (c - a b x) T_{q} \right\}, \allowdisplaybreaks \\
H_{4} (a, b, c; x) &:= T_{q}, \allowdisplaybreaks \\
B_{1} (a, b, c; x) &:= \left\{(q - a) (c - a) \right\}^{-1} 
\left\{c q - a (q + c) + a^{2} x + a (c - a b x) T_{q} \right\}, \allowdisplaybreaks \\
B_{2} (a, b, c; x) &:= \left\{(q - b) (c - b) \right\}^{-1} 
\left\{c q - b (q + c) + b^{2} x + b (c - a b x) T_{q} \right\}, \allowdisplaybreaks \\
B_{3} (a, b, c; x) &:= 1 - c q^{-1} T_{q}, \allowdisplaybreaks \\
B_{4} (a, b, c; x) &:= (q - x)^{-1} 
\left\{c + q - (a + b) x - (c - a b x) T_{q} \right\}. 
\end{align*}
Then, the operators $H_{j}$ ($j = 1, 2, 3$) 
increase the parameters $a, b$, and $c$ by $q$ times, respectively, 
while the operators $B_{j}$ ($j = 1, 2, 3$) decrease these parameters by $q$ times. 
Also, the operators $H_{4}$ and $B_{4}$ shift the variable $x$ by $q$ and $q^{-1}$ times, respectively. 
Hence, we call these eight operators ``contiguity operators.'' 
The definitions of $H_{1}, H_{2}$, and $B_{3}$ are easily derived from 
Heine's \cite[p.~287, Formula~2 and p.~292, Formulas~42--44]{Heine} 
three-term relations 
and the definition of $H_{4}$ immediately comes from the definition of $T_{q}$. 
Using the method described in \cite[p.~46, Remark~2.1.4]{IKSY}, we are able to derive 
$B_{1}, B_{2}, H_{3}$, and $B_{4}$ from $H_{1}, H_{2}, B_{3}$, and $H_{4}$, respectively. 

In fact, performing direct calculations, we obtain the following lemma: 
\begin{lem}\label{lem:3-2}
Let us write $y_{i} (a, b; c; x),\ y_{i} (a q^{\pm 1}, b; c; x),\ y_{i} (a, b q^{\pm 1}; c; x)$, 
$y_{i} (a, b; c q^{\pm 1}; x)$, and $y_{i} (a, b; c; x q^{\pm 1})$ 
as $y_{i},\ y_{i} (a q^{\pm 1}),\ y_{i} (b q^{\pm 1}),\ y_{i} (c q^{\pm 1})$, 
and $y_{i} (x q^{\pm 1})$, respectively. 
Let $H_{j}$ and $B_{j}$ denote $H_{j} (a, b, c; x)$ and $B_{j} (a, b, c; x)$, respectively. 
Then, for $i = 1$ and $2$, we have 
\begin{align*}
H_{1} y_{i} &= y_{i} (a q), 
& 
B_{1} y_{i} &= y_{i} (a q^{-1}), \allowdisplaybreaks \\
H_{2} y_{i} &= y_{i} (b q), 
& 
B_{2} y_{i} &= y_{i} (b q^{-1}), \allowdisplaybreaks \\
H_{3} y_{i} &= y_{i} (c q), 
& 
B_{3} y_{i} &= y_{i} (c q^{-1}). 
\end{align*}
Also, for $i = 3$ and $4$, we have 
\begin{align*}
H_{1} y_{i} &= - a \cdot y_{i} (a q), 
& 
B_{1} y_{i} &= - q a^{-1} \cdot y_{i} (a q^{-1}), \allowdisplaybreaks \\
H_{2} y_{i} &= - b \cdot y_{i} (b q), 
& 
B_{2} y_{i} &= - q b^{-1} \cdot y_{i} (b q^{-1}), \allowdisplaybreaks \\
H_{3} y_{i} &= - c^{-1} \cdot y_{i} (c q), 
& 
B_{3} y_{i} &= - c q^{-1} \cdot y_{i} (c q^{-1}). 
\end{align*}
Moreover, for $i = 1, 2, 3$, and $4$, we have 
\begin{align*}
H_{4} y_{i} &= y_{i} (x q), 
&
B_{4} y_{i} &= y_{i} (x q^{-1}). 
\end{align*}
\end{lem}
Let $S_{q} (a, b, c; x)$ denote the solution space of $E_{q} (a, b, c; x)$ 
on a simply connected domain in $\mathbb{C} \,\backslash \left\{0 \right\}$. 
Then, from Lemmas~$\ref{lem:3-1}$ and $\ref{lem:3-2}$, 
under the assumption $(\ref{assump})$, 
combining $H_{j}$ and $B_{j}$ ($j = 1, 2, 3, 4$), 
we obtain a linear isomorphism 
\begin{align*}
\theta (k, l, m; n) : S_{q} (a, b, c; x) \to S_{q} (a q^{k}, b q^{l}, c q^{m}; x q^{n}) 
\end{align*}
for any integers $k, l, m$, and $n$. 
Note that the definition of $\theta (k, l, m; n)$ does not depend on the order of composition of contiguity operators. 
Here, we set the order as follows: 
\begin{align*}
S_{q} (a, b, c; x) \to \dotsm 
&\to S_{q} (a q^{k}, b, c; x) \to \dotsm \to S_{q} (a q^{k}, b q^{l}, c; x) \to \dotsm \\
&\to S_{q} (a q^{k}, b q^{l}, c q^{m}; x) \to \dotsm \to S_{q} (a q^{k}, b q^{l}, c q^{m}; x q^{n}). 
\end{align*}

\subsection{Expressions for each $\widetilde{Q}$ and $\widetilde{R}$ as a ratio of infinite series}\quad 
We express each $\widetilde{Q}$ and $\widetilde{R}$ as a ratio of infinite series 
defined as a sum of products of $y_{i}$ ($i = 1, 2, 3, 4$). 

Put $\Delta := x^{-1} (1 - T_{q})$. 
By direct calculations, we obtain the following lemma: 
\begin{lem}\label{lem:3-3}
Let $y_{i} (a, b, c)$ denote $y_{i} (a, b; c; x)$. Then, we have 
\begin{align*}
\Delta y_{i} (a, b, c) = 
\begin{cases}
y_{i} (a q, b q, c q), & i = 1, 2, \\
- a b c^{-1} \cdot y_{i} (a q, b q, c q), & i = 3, 4. 
\end{cases}
\end{align*}
\end{lem}

Let $\mathbb{Q} (a, b, c, q, x)$ denote the field generated by $a, b, c, q$, and $x$ over $\mathbb{Q}$, 
and let us write the ring of polynomials in $T_{q}$ over $\mathbb{Q} (a, b, c, q, x)$ 
as $\mathbb{Q} (a, b, c, q, x) [T_{q}]$, where $T_{q}$ and $x$ are not commutative because $T_{q} x = q x T_{q}$. 
Then, regarding $\theta (k, l, m; n)$ as an element of $\mathbb{Q} (a, b, c, q, x) [T_{q}]$, 
we express $\theta (k, l, m; n)$ as 
\begin{align*}
\theta (k, l, m; n) = p (T_{q}) \cdot L + \hat{Q} \cdot T_{q} + \hat{R}, 
\end{align*}
where $\hat{Q}, \hat{R} \in \mathbb{Q} (a, b, c, q, x)$ and 
$p (T_{q}) \in \mathbb{Q} (a, b, c, q, x) [T_{q}]$. 
From $T_{q} = 1 - x \Delta$, this expression can be rewritten as 
\begin{align}\label{theta}
\theta (k, l, m; n) = p (T_{q}) \cdot L - x \hat{Q} \cdot \Delta + (\hat{Q} + \hat{R}). 
\end{align}
From Lemmas~$\ref{lem:3-1},\ \ref{lem:3-2}$, and $\ref{lem:3-3}$, 
operating $(\ref{theta})$ on $y_{1} (a, b; c; x)$ gives 
\begin{align}\label{gen.3tr''}
\f{a q^{k}}{b q^{l}}{c q^{m}}{x q^{n}} = - x \hat{Q} \cdot \f{a q}{b q}{c q}{x} + (\hat{Q} + \hat{R}) \cdot \f{a}{b}{c}{x}. 
\end{align}
Since the pair $(\widetilde{Q}, \widetilde{R})$ satisfying $(\ref{gen.3tr'})$ 
is unique as a consequence of Section~$3. 1$, comparing $(\ref{gen.3tr''})$ with $(\ref{gen.3tr'})$, we find that 
$- x \hat{Q} = \widetilde{Q}$ and $\hat{Q} + \hat{R} = \widetilde{R}$. 
Namely, we obtain 
\begin{align}\label{theta'}
\theta (k, l, m; n) = p (T_{q}) \cdot L + \widetilde{Q} \cdot \Delta + \widetilde{R}. 
\end{align}
From Lemmas~$\ref{lem:3-1},\ \ref{lem:3-2}$, and $\ref{lem:3-3}$, 
operating $(\ref{theta'})$ on $y_{i} (a, b; c; x)$ ($i = 1, 2, 3, 4$) gives 
\begin{align}
\y{i}{a q^{k}}{b q^{l}}{c q^{m}}{x q^{n}} 
&= \widetilde{Q} \cdot \y{i}{a q}{b q}{c q}{x} + \widetilde{R} \cdot \y{i}{a}{b}{c}{x}, \quad i = 1, 2, 
\label{gen.3tr_y1,y2} \allowdisplaybreaks \\
\lambda \cdot \y{i}{a q^{k}}{b q^{l}}{c q^{m}}{x q^{n}} 
&= - a b c^{-1} \widetilde{Q} \cdot \y{i}{a q}{b q}{c q}{x} + \widetilde{R} \cdot \y{i}{a}{b}{c}{x}, \quad i = 3, 4, 
\label{gen.3tr_y3,y4}
\end{align}
where $\lambda := (-1)^{k + l - m} a^{k} b^{l} c^{-m} q^{\left\{k (k - 1) + l (l - 1) - m (m - 1) \right\} / 2}$. 
Solving two equations $(\ref{gen.3tr_y1,y2})$ for $\widetilde{Q}$ and $\widetilde{R}$, we have 
\begin{align}
\widetilde{Q} = \widetilde{Q} (k, l, m, n) 
&= \Y{k}{l}{m}{n}{a}{b}{c}{x} \,\Bigg/\, \Y{1}{1}{1}{0}{a}{b}{c}{x}, \label{gen.tQ,Y} \allowdisplaybreaks \\
\widetilde{R} = \widetilde{R} (k, l, m, n) 
&= - \Y{k - 1}{l - 1}{m - 1}{n}{a q}{b q}{c q}{x} \,\Bigg/\, \Y{1}{1}{1}{0}{a}{b}{c}{x}, \label{gen.tR,Y}
\end{align}
where $Y$ is the infinite series defined by 
\begin{align}\label{gen.def:Y}
\Y{k}{l}{m}{n}{a}{b}{c}{x} := 
\y{1}{a q^{k}}{b q^{l}}{c q^{m}}{x q^{n}} 
\y{2}{a}{b}{c}{x} 
- \y{2}{a q^{k}}{b q^{l}}{c q^{m}}{x q^{n}} 
\y{1}{a}{b}{c}{x}. 
\end{align}
On the other hand, solving two equations $(\ref{gen.3tr_y3,y4})$ for $\widetilde{Q}$, we have 
\begin{align}
\widetilde{Q} = \widetilde{Q} (k, l, m, n) 
&= - \lambda c (a b)^{-1} \cdot 
\Z{k}{l}{m}{n}{a}{b}{c}{x} \,\Bigg/\, \Z{1}{1}{1}{0}{a}{b}{c}{x}, \label{gen.tQ,Z} 
\end{align}
where $\widetilde{Y}$ is the infinite series defined by 
\begin{align}\label{gen.def:Z}
\Z{k}{l}{m}{n}{a}{b}{c}{x} := 
\y{3}{a q^{k}}{b q^{l}}{c q^{m}}{x q^{n}} 
\y{4}{a}{b}{c}{x} 
- \y{4}{a q^{k}}{b q^{l}}{c q^{m}}{x q^{n}} 
\y{3}{a}{b}{c}{x}. 
\end{align}

In the next two sections, we use $(\ref{gen.tQ,Y})$--$(\ref{gen.def:Y})$ for proving Theorem~$\ref{gen.main}$, 
and use $(\ref{gen.tQ,Z})$ and $(\ref{gen.def:Z})$ for proving Proposition~$\ref{gen.main2}$. 

\begin{remark}\label{rem2}
The Casoratians noted in Remark~$\ref{rem1}$ 
are appearing in the denominators of $(\ref{gen.tQ,Y}),\ (\ref{gen.tR,Y})$, and $(\ref{gen.tQ,Z})$ as follows: 
\begin{align*}
\Y{1}{1}{1}{0}{a}{b}{c}{x} 
&= 
- \det 
\left( 
\begin{array}{rr}
y_{1} & y_{2} \\
\Delta y_{1} & \Delta y_{2} 
\end{array}
\right) 
= x^{-1} \cdot 
\det 
\left( 
\begin{array}{rr}
y_{1} & y_{2} \\
T_{q} y_{1} & T_{q} y_{2} 
\end{array}
\right), 
\allowdisplaybreaks \\
\Z{1}{1}{1}{0}{a}{b}{c}{x} 
&= c (a b)^{-1} \cdot 
\det 
\left( 
\begin{array}{rr}
y_{3} & y_{4} \\
\Delta y_{3} & \Delta y_{4} 
\end{array}
\right) 
= - c (a b x)^{-1} \cdot 
\det 
\left( 
\begin{array}{rr}
y_{3} & y_{4} \\
T_{q} y_{3} & T_{q} y_{4} 
\end{array}
\right). 
\end{align*}
\end{remark}

\subsection{Expressions for $Q$ and $R$}\quad 
We complete the proof of Theorem~$\ref{gen.main}$. 

To prove Theorem~$\ref{gen.main}$, we give four lemmas. 
From Lemma~$\ref{Gasper.tf1,2}$, we obtain the following lemma: 
\begin{lem}\label{gen.lem:(i)-(iv)}
For any integers $k, l, m$, and $n$, with $k \leq l$, the following statements are true: 
\begin{align}
\tag{i} k + l - m + n \geq 0, \; n \geq 0 \; &\Rightarrow \; 
\sum_{i = 0}^{n} \frac{(q^{-n})_{i}}{(q)_{i}} q^{n i} (A_{j - i} - B_{j - i - m}) = 0, 
\; j \geq l + n; \allowdisplaybreaks \\
\tag{ii} k + l - m + n \geq 0, \; n \leq 0 \; &\Rightarrow \; 
A_{j} - B_{j - m} = 0, \; j \geq l; \allowdisplaybreaks \\
\tag{iii} k + l - m + n \leq 0, \; n \geq 0 \; &\Rightarrow \; 
\widetilde{A}_{j} - \widetilde{B}_{j - m} = 0, \; j \geq m - k; \allowdisplaybreaks \\
\tag{iv} k + l - m + n \leq 0, \; n \leq 0 \; &\Rightarrow \; 
\sum_{i = 0}^{-n} \frac{(q^{n})_{i}}{(q)_{i}} (\widetilde{A}_{j - i} - \widetilde{B}_{j - i - m}) = 0, 
\; j \geq m - k - n. 
\end{align}
\end{lem}
\begin{proof}
Using $(\ref{Gasper.tf2})$, we prove (i). 
Replacing $m, n_{1}, n_{2}, s, a, b, c_{1}, c_{2}$ in the $r = 2$ case of $(\ref{Gasper.tf2})$ 
by $j, j - k, j - l, n, q^{m - j}, c q^{m - j - 1}, c q^{m - j} / a, c q^{m - j} / b$, respectively, 
we find that if 
$k + l - m + n \geq 0$ and $\min\left\{j - k, j - l \right\} \geq n \geq 0$, then 
\begin{align*}
&\sum_{i = 0}^{n} \frac{(q^{-n})_{i}}{(q)_{i}} q^{n i} 
\frac{(a q / c)_{j - i - m} (b q / c)_{j - i - m}}{(q)_{j - i - m} (q^{2} / c)_{j - i - m}} 
\frac{(q^{2} / c)_{- m - 1}}{(a q / c)_{k - m} (b q / c)_{l - m}} \\
&\quad \times 
\P{q^{- (j - i - m)}}{c q^{- (j - i - m) - 1}}{c q^{m - k} / a}{c q^{m - l} / b}{c q^{m}}{c q^{- (j - i - m)} / a}{c q^{- (j - i - m)} / b}{q^{k + l - m + n + 1}} 
\allowdisplaybreaks \\
&= - \sum_{i = 0}^{n} \frac{(q^{-n})_{i}}{(q)_{i}} q^{n i} 
\frac{(c)_{m - (j - i) - 1} (c q^{m - (j - i) - 1})^{1 - n}}{(q^{- (j - i)})_{j - i} (a)_{k - (j - i)} (b)_{l - (j - i)}} 
\P{q^{- (j - i)}}{c q^{m - (j - i) - 1}}{a}{b}{c}{a q^{k - (j - i)}}{b q^{l - (j - i)}}{q^{1 - n}}. 
\end{align*}
Moreover, multiplying both sides of this by $(a q / c)_{k - m} (b q / c)_{l - m} / (q^{2} / c)_{- m - 1}$, we have 
\begin{align*}
\sum_{i = 0}^{n} \frac{(q^{-n})_{i}}{(q)_{i}} q^{n i} B_{j - i - m} = \sum_{i = 0}^{n} \frac{(q^{-n})_{i}}{(q)_{i}} q^{n i} A_{j - i}. 
\end{align*}
Thus (i) is proved. 

In almost the same way, using $(\ref{Gasper.tf1})$, we are able to prove (ii) and (iii), 
and using $(\ref{Gasper.tf2})$, we are able to prove (iv). 
\end{proof}

To be used below, we introduce a formula: 
\begin{align}\label{q-binom'}
\sum_{i = 0}^{n} \frac{(q^{-n})_{i}}{(q)_{i}} x^{i} = (x q^{-n})_{n}, \quad n = 0, 1, \dotsc. 
\end{align}
This is a special case of the $q$-binomial theorem 
\begin{align*}
{}_{1}\phi_{0} \biggl(\genfrac..{0pt}{}{a}{\mbox{--}};q,\,x\biggr) = 
\sum_{i = 0}^{\infty} \frac{(a)_{i}}{(q)_{i}} x^{i} = \frac{(a x)_{\infty}}{(x)_{\infty}}, 
\quad \lvert x \rvert < 1. 
\end{align*}
(See \cite[pp.~488--490, Theorem~10.2.1]{AAR} and \cite[Section~1.3]{GR} 
for simple proofs of the $q$-binomial theorem.) 
We also use the following product formula: 
\begin{align}\label{2phi1,2phi1}
&\p{a}{b}{c}{g x} \p{d}{e}{f}{h x} \\
&= \sum_{j = 0}^{\infty} \frac{(a)_{j} (b)_{j}}{(q)_{j} (c)_{j}} g^{j} 
\P{q^{-j}}{q^{1 - j} / c}{d}{e}{q^{1 - j} / a}{q^{1 - j} / b}{f}{\frac{q c h}{a b g}} 
x^{j}, \quad \lvert g x \rvert < 1,\, \lvert h x \rvert < 1. \nonumber 
\end{align}
Collecting the coefficients of $x^{j}$ on the left-hand side gives this formula. 

From Lemma~$\ref{gen.lem:(i)-(iv)},\ (\ref{q-binom'})$, and $(\ref{2phi1,2phi1})$, we obtain the following lemma: 
\begin{lem}\label{gen.lem:main}
When $k \leq l$, the polynomial $P$ defined in Theorem~$\ref{gen.main}$ can be rewritten in the following form: 
\begin{align}\label{gen.P_inf}
&P \c{k}{l}{m}{n}{a}{b}{c}{x} \\
&= 
\begin{cases}
\displaystyle
(x)_{\max\left\{n, 0 \right\}}\, x^{\max\left\{m, 0 \right\}} 
\left(x^{-m} \sum_{j = 0}^{\infty} A_{j} x^{j} - \sum_{j = 0}^{\infty} B_{j} x^{j} \right), & 
k + l - m + n \geq 0, \\
\displaystyle
(x q^{\min\left\{n, 0 \right\}})_{-\min\left\{n, 0 \right\}}\, x^{\max\left\{m, 0 \right\}} 
\left(x^{-m} \sum_{j = 0}^{\infty} \widetilde{A}_{j} x^{j} - \sum_{j = 0}^{\infty} \widetilde{B}_{j} x^{j} \right), & 
k + l - m + n < 0. 
\end{cases}\nonumber
\end{align}
Moreover, each of the infinite series $\sum_{j = 0}^{\infty} A_{j} x^{j},\ \sum_{j = 0}^{\infty} B_{j} x^{j}$, 
$\sum_{j = 0}^{\infty} \widetilde{A}_{j} x^{j}$, and $\sum_{j = 0}^{\infty} \widetilde{B}_{j} x^{j}$ 
can be rewritten as a product of two ${}_{2}\phi_{1}$ series as follows, respectively: 
\begin{align*}
&\frac{(a q / c)_{k - m} (b q / c)_{l - m} (c)_{m}}{(q^{2} / c)_{-m} (a)_{k} (b)_{l}} (c q^{m - 1})^{-n} 
\p{q^{1 - k} / a}{q^{1 - l} / b}{q^{2 - m} / c}{\frac{a b q^{k + l - m + n}}{c} x} \p{a}{b}{c}{x}, \allowdisplaybreaks \\
&\p{c q^{m - k} / a}{c q^{m - l} / b}{c q^{m}}{\frac{a b q^{k + l - m + n}}{c} x} 
\p{a q / c}{b q / c}{q^{2} / c}{x}, \allowdisplaybreaks \\
&\frac{(a q / c)_{k - m} (b q / c)_{l - m} (c)_{m}}{(q^{2} / c)_{-m} (a)_{k} (b)_{l}} (c q^{m - 1})^{-n} 
{}_{2}\phi_{1}\biggl(\genfrac..{0pt}{}{a q^{k + 1 - m} / c,\,b q^{l + 1 - m} / c}{q^{2 - m} / c};q, x q^{n}\biggr) 
{}_{2}\phi_{1}\biggl(\genfrac..{0pt}{}{c / a,\,c / b}{c};q, \frac{a b}{c} x\biggr), \allowdisplaybreaks \\
&\p{a q^{k}}{b q^{l}}{c q^{m}}{x q^{n}} \p{q / a}{q / b}{q^{2} / c}{\frac{a b}{c} x}. 
\end{align*}
\end{lem}
\begin{proof}
From Lemma~$\ref{gen.lem:(i)-(iv)}$, we obtain 
\begin{align}\label{lem:main-1}
&P \c{k}{l}{m}{n}{a}{b}{c}{x} \\
&= 
\begin{cases}
\displaystyle\sum_{j = 0}^{\infty} \sum_{i = 0}^{\max\left\{n, 0 \right\}} \frac{\left(q^{-n} \right)_{i}}{(q)_{i}} q^{n i} 
\left(A_{j - i + m - \max\left\{m, 0 \right\}} - B_{j - i - \max\left\{m, 0 \right\}} \right) x^{j}, 
& k + l - m + n \geq 0, \\
\displaystyle\sum_{j = 0}^{\infty} \sum_{i = 0}^{-\min\left\{n, 0 \right\}} \frac{\left(q^{n} \right)_{i}}{(q)_{i}} 
\left(\widetilde{A}_{j - i + m - \max\left\{m, 0 \right\}} - \widetilde{B}_{j - i - \max\left\{m, 0 \right\}} \right) x^{j}, 
& k + l - m + n < 0. 
\end{cases}\nonumber
\end{align}
Also, from $(\ref{q-binom'})$, we have
\begin{align}\label{lem:main-2}
\sum_{i = 0}^{\max\left\{n, 0 \right\}} \frac{\left(q^{-n} \right)_{i}}{(q)_{i}} (x q^{n})^{i} 
= (x)_{\max\left\{n, 0 \right\}}, \quad 
\sum_{i = 0}^{-\min\left\{n, 0 \right\}} \frac{\left(q^{n} \right)_{i}}{(q)_{i}} x^{i} 
= (x q^{\min\left\{n, 0 \right\}})_{- \min\left\{n, 0 \right\}}. 
\end{align}
Using $(\ref{lem:main-1}),\ (\ref{lem:main-2})$, 
and the fact that $A_{j}, B_{j}, \widetilde{A}_{j}$, and $\widetilde{B}_{j}$ 
are equal to zero for any negative integer $j$, 
we obtain $(\ref{gen.P_inf})$. 
Also, it follows from $(\ref{2phi1,2phi1})$ that the statement following $(\ref{gen.P_inf})$ is true. 
Thus the lemma is proved. 
\end{proof}

Heine~\cite[p.~325, Formula~XVIII]{Heine} obtained the following $q$-Euler transformation formula: 
\begin{align}\label{Heine}
\p{a}{b}{c}{x} = \frac{(a b x / c)_{\infty}}{(x)_{\infty}} \p{c / a}{c / b}{c}{\frac{a b}{c} x}. 
\end{align}

From Lemma~$\ref{gen.lem:main}$ and $(\ref{Heine})$, we obtain the following lemma: 
\begin{lem}\label{gen.lem:Y,P}
When $k \leq l$, the series $Y$ defined by $(\ref{gen.def:Y})$ can be expressed as 
\begin{align*}
\Y{k}{l}{m}{n}{a}{b}{c}{x} 
= - \lambda_{1} \frac{(a)_{k} (b)_{l}}{(c)_{m}} 
x^{1 - \gamma - \max\left\{m, 0 \right\}} 
\frac{(a b q^{\max\left\{k + l - m + n, 0\right\}} x / c)_{\infty}}{(x q^{\min\left\{n, 0 \right\}})_{\infty}} 
P \c{k}{l}{m}{n}{a}{b}{c}{x}, 
\end{align*}
where 
$\lambda_{1} := 
(q)_{\infty}^{2} (c)_{\infty} (q^{2}/ c)_{\infty} \big/ \left\{(a)_{\infty} (b)_{\infty} (a q / c)_{\infty} (b q / c)_{\infty} \right\}$. 
\end{lem}
\begin{proof}
In the cases of $k + l - m + n \geq 0$, 
we apply $(\ref{Heine})$ to $y_{i} (a q^{k}, b q^{l}; c q^{m}; x q^{n})$ ($i = 1, 2$) in $(\ref{gen.def:Y})$. 
On the other hand, 
in the cases of $k + l - m + n < 0$, 
we apply $(\ref{Heine})$ to $y_{i} (a, b; c; x)$ ($i = 1, 2$) in $(\ref{gen.def:Y})$. 
Then, from Lemma~$\ref{gen.lem:main}$, we obtain Lemma~$\ref{gen.lem:Y,P}$. 
\end{proof}

From Lemma~$\ref{gen.lem:Y,P}$ and the definition of $P$, we have 
\begin{align}\label{Y(1,1,1,0)}
\Y{1}{1}{1}{0}{a}{b}{c}{x} 
= \lambda_{1} \frac{q - c}{c} x^{- \gamma} 
\frac{(a b q x / c)_{\infty}}{(x)_{\infty}}. 
\end{align}
Therefore, from $(\ref{gen.tQ,Y}),\ (\ref{gen.tR,Y}),\ (\ref{Y(1,1,1,0)})$, and Lemma~$\ref{gen.lem:Y,P}$, 
we obtain the following lemma: 
\begin{lem}\label{gen.tQ,tR}
Assume that $k \leq l$. Then, the coefficients of $(\ref{gen.3tr'})$ can be expressed as 
\begin{align*}
\widetilde{Q} (k, l, m, n) 
&= - \frac{c (a)_{k} (b)_{l}}{(q - c) (c)_{m}} 
\frac{x^{1 - \max\left\{m, 0 \right\}} (x)_{\min\left\{n, 0 \right\}}}{(a b q x / c)_{\max\left\{k + l - m + n, 0 \right\} - 1}} 
P \c{k}{l}{m}{n}{a}{b}{c}{x}, \allowdisplaybreaks \\
\widetilde{R} (k, l, m, n) 
&= - \frac{(a)_{k} (b)_{l}}{(c)_{m}} 
\frac{x^{- \max\left\{m - 1, 0 \right\}} (x)_{\min\left\{n, 0 \right\}}}{(a b q x / c)_{\max\left\{k + l - m + n - 1, 0 \right\}}} 
P \c{k - 1}{l - 1}{m - 1}{n}{a q}{b q}{c q}{x}. 
\end{align*}
\end{lem}

From $(\ref{gen.Q,tQ}),\ (\ref{gen.R,tR})$, and Lemma~$\ref{gen.tQ,tR}$, 
we are able to complete the proof of Theorem~$\ref{gen.main}$.

\subsection{Alternative expressions for $Q$ and $R$}\quad 
We complete the proof of Proposition~$\ref{gen.main2}$. 

Let $\widetilde{P}$ be the polynomial in $x$ defined by 
\begin{align*}
\widetilde{P} \c{k}{l}{m}{n}{a}{b}{c}{x} := 
\begin{cases}
\displaystyle\sum_{j = 0}^{d} \sum_{i = 0}^{\max\left\{n, 0 \right\}} \frac{\left(q^{-n} \right)_{i}}{(q)_{i}} q^{i} 
\left(C_{j - i} - D_{j - i + k - l} \right) x^{d - j}, 
& k + l - m + n \geq 0, \\
\displaystyle\sum_{j = 0}^{d} \!\sum_{i = 0}^{-\min\left\{n, 0 \right\}} \!\frac{\left(q^{n} \right)_{i}}{(q)_{i}} q^{(1 - n) i} 
\left(\widetilde{C}_{j - i} - \widetilde{D}_{j - i + k - l} \right) x^{d - j}, 
& k + l - m + n < 0, 
\end{cases}
\end{align*}
where $d, C_{j}, D_{j}, \widetilde{C}_{j}$, and $\widetilde{D}_{j}$ are defined as in Proposition~$\ref{gen.main2}$. 

To prove Proposition~$\ref{gen.main2}$, we give four lemmas. 
By using Lemma~$\ref{Gasper.tf1,2}$ as in the proof of Lemma~$\ref{gen.lem:(i)-(iv)}$, 
we obtain the following lemma: 
\begin{lem}\label{gen.lem:(i)-(iv).2}
For any integers $k, l, m$, and $n$, with $k \leq l$, the following statements are true: 
\begin{align*}
k + l - m + n \geq 0, \; n \geq 0 \; &\Rightarrow 
\sum_{i = 0}^{n} \frac{(q^{-n})_{i}}{(q)_{i}} q^{i} (C_{j - i} - D_{j - i + k - l}) = 0, 
\; j \geq \max\left\{l, l - m \right\} + n; 
\allowdisplaybreaks \\
k + l - m + n \geq 0, \; n \leq 0 \; &\Rightarrow \; 
C_{j} - D_{j + k - l} = 0, \; j \geq \max\left\{l, l - m \right\}; 
\allowdisplaybreaks \\
k + l - m + n \leq 0, \; n \geq 0 \; &\Rightarrow \; 
\widetilde{C}_{j} - \widetilde{D}_{j + k - l} = 0, \; j \geq \max\left\{-k, m - k \right\}; 
\allowdisplaybreaks \\
k + l - m + n \leq 0, \; n \leq 0 \; &\Rightarrow 
\sum_{i = 0}^{-n} \frac{(q^{n})_{i}}{(q)_{i}} q^{(1 - n) i} (\widetilde{C}_{j - i} - \widetilde{D}_{j - i + k - l}) = 0, 
\; j \geq \max\left\{-k, m - k \right\} - n. 
\end{align*}
\end{lem}

Also, in the same way as in the proof of Lemma~$\ref{gen.lem:main}$, 
by using $(\ref{q-binom'}),\ (\ref{2phi1,2phi1})$, and Lemma~$\ref{gen.lem:(i)-(iv).2}$, 
we obtain the following lemma: 
\begin{lem}\label{gen.tP_inf}
When $k \leq l$, the polynomial $\widetilde{P}$ can be rewritten in the following form: 
\begin{align*}
&\widetilde{P} \c{k}{l}{m}{n}{a}{b}{c}{x} \\
&= x^{d} 
\begin{cases}
\displaystyle \frac{1}{(q / x)_{-\max\left\{n, 0 \right\}}} \mu_{1} \cdot 
\p{q^{1 - l} / b}{c q^{m - l} / b}{a q^{k - l + 1} / b}{\frac{q^{1 - n}}{x}} \p{b}{b q / c}{b q / a}{\frac{c q}{a b x}} \\
\displaystyle \quad - \frac{1}{(q / x)_{-\max\left\{n, 0 \right\}}} x^{k - l} \mu_{2} \cdot 
\p{q^{1 - k} / a}{c q^{m - k} / a}{b q^{l - k + 1} / a}{\frac{q^{1 - n}}{x}} \p{a}{a q / c}{a q / b}{\frac{c q}{a b x}}, \\
\hspace{89mm} k + l - m + n \geq 0, \allowdisplaybreaks \\
\displaystyle (q / x)_{-\min\left\{n, 0 \right\}} \mu_{1} \cdot 
\p{a q^{k}}{a q^{k - m + 1} / c}{a q^{k - l + 1} / b}{\frac{c q^{m - k - l - n + 1}}{a b x}} 
\p{q / a}{c / a}{b q / a}{\frac{q}{x}} \\
\displaystyle \quad - (q / x)_{-\min\left\{n, 0 \right\}} x^{k - l} \mu_{2} \cdot 
\p{b q^{l}}{b q^{l - m + 1} / c}{b q^{l - k + 1} / a}{\frac{c q^{m - k - l - n + 1}}{a b x}} 
\p{q / b}{c / b}{a q / b}{\frac{q}{x}}, \\
\hspace{89mm} k + l - m + n < 0, 
\end{cases}
\end{align*}
where $d,\ \mu_{1}$ and $\mu_{2}$ are defined as in Proposition~$\ref{gen.main2}$. 
\end{lem}

Moreover, in the same way as in the proof of Lemma~$\ref{gen.lem:Y,P}$, 
by applying $(\ref{Heine})$ to $(\ref{gen.def:Z})$ and 
using Lemma~$\ref{gen.tP_inf}$, we obtain the following lemma: 
\begin{lem}\label{gen.lem:Z,tP}
The series $\widetilde{Y}$ defined by $(\ref{gen.def:Z})$ can be expressed as 
\begin{align*}
\Z{k}{l}{m}{n}{a}{b}{c}{x} 
&= \lambda_{2} 
\frac{x^{- \alpha - \beta - k - d} (q^{1 - \min\left\{n, 0 \right\}} / x)_{\infty}}{(c q^{1 - \max\left\{k + l - m + n, 0 \right\}} / (a b x))_{\infty}} 
\widetilde{P} \c{k}{l}{m}{n}{a}{b}{c}{x}, 
\end{align*}
where 
$\lambda_{2} := 
(a b)^{\gamma - \alpha - \beta + 1} (q)_{\infty}^{2} (a q / b)_{\infty} (b q / a)_{\infty} 
\big/ \left\{(a)_{\infty} (b)_{\infty} (a q / c)_{\infty}(b q / c)_{\infty} \right\}$. 
\end{lem}

From Lemma~$\ref{gen.lem:Z,tP}$ and the definition of $\widetilde{P}$, we obtain 
\begin{align}\label{tY(1,1,1,0)}
\Z{1}{1}{1}{0}{a}{b}{c}{x} 
= \lambda_{2} \frac{c (b - a)}{a^{2} b^{2}} 
\frac{x^{- \alpha - \beta - 1} (q / x)_{\infty}}{(c / (a b x))_{\infty}}. 
\end{align}
Therefore, 
from $(\ref{gen.Q,tQ}),\ (\ref{gen.tQ,Z}),\ (\ref{tY(1,1,1,0)})$, and Lemma~$\ref{gen.lem:Z,tP}$, 
we obtain the following lemma: 
\begin{lem}\label{gen.lem:Q,tP}
For any integers $k, l, m$, and $n$, with $k \leq l$, the coefficient $Q$ of $(\ref{gen.3tr})$ can be expressed as 
\begin{align*}
Q = Q (k, l, m, n) = - \mu \frac{(1 - a) (1 - b) c}{(q - c) (1 - c)} 
\frac{x^{1 - \max\left\{m, 0 \right\}} (x)_{\min\left\{n, 0 \right\}}}{(a b q x / c)_{\max\left\{k + l - m + n, 0 \right\} - 1}} 
\widetilde{P} \c{k}{l}{m}{n}{a}{b}{c}{x}, 
\end{align*}
where $\mu$ is defined as in Proposition~$\ref{gen.main2}$. 
\end{lem}

From the uniqueness of the pair $(Q, R)$ proved in Section~$3. 1$, 
we find that $Q$ appearing in Theorem~$\ref{gen.main}$ and 
$Q$ appearing in Lemma~$\ref{gen.lem:Q,tP}$ are equal to each other. 
Therefore, by comparing the expressions for $Q$ given in Theorem~$\ref{gen.main}$ and Lemma~$\ref{gen.lem:Q,tP}$, 
we have $P = \mu \widetilde{P}$. 
This completes the proof of Proposition~$\ref{gen.main2}$.

\begin{remark}
We consider the degree of $P$. 
From Proposition~$\ref{gen.main2}$, we find that 
for any integers $k, l, m$, and $n$, with $k \leq l$, 
the degree of the polynomial $P$ is no more than $d$, 
where $d := \max\left\{k + l - m + n, 0 \right\} + \max\left\{m, 0 \right\} - \min\left\{n, 0 \right\} - k - 1$, 
and the coefficient of $x^{d}$ in $P$ 
equals $\mu (C_{0} - D_{0}) \rvert_{k = l}$ if $k = l$ ; $\mu C_{0}$ if $k < l$. 
Namely, the coefficient of $x^{d}$ in $P$ equals 
\begin{align*}
\mu 
\begin{cases}
\, \displaystyle\left(\frac{c}{a b} \right)^{k} q^{k (m - k - l - n + 1)} 
\left\{a^{m - 2k - n} (a)_{k} (a q / c)_{k - m} - b^{m - 2k - n} (b)_{k} (b q / c)_{k - m} \right\}, & k = l, \\
\, \displaystyle\left(\frac{c}{a b} \right)^{k} q^{k (m - k - l - n + 1)} a^{m - k - l - n} 
\frac{(a)_{k} (a q / c)_{k - m}}{(a q / b)_{k - l}}, & k < l, 
\end{cases}
\end{align*}
where $\mu$ is defined as in Proposition~$\ref{gen.main2}$. 
\end{remark}

\subsection{Relation between $Q$ and $R$}\quad 
We prove Corollary~$\ref{cor}$. 

From Theorem~$\ref{gen.main}$, we have 
\begin{align*}
&Q (k - 1, l - 1, m - 1, n) \big\rvert _{(a, b, c) \mapsto (a q, b q, c q)} \\
&\quad = - \frac{(1 - a q) (1 - b q) c}{(1 - c) (1 - c q)} 
\frac{x^{1 - \max\left\{m - 1, 0 \right\}} (x)_{\min\left\{n, 0 \right\}}}{(a b q^{2} x / c)_{\max\left\{k + l - m + n - 1, 0 \right\} - 1}} 
P \c{k - 1}{l - 1}{m - 1}{n}{a q}{b q}{c q}{x} \\
&\quad = \frac{(1 - a q) (1 - b q) x (c - a b q x)}{(1 - c) (1 - c q)} R (k, l, m, n). 
\end{align*}
Thus Corollary~$\ref{cor}$ is proved.

\medskip
\thanks{\bf{Acknowledgments}}
We are deeply grateful to Prof. Hiroyuki Ochiai for helpful comments. 
Also, we would like to thank Naoya Yamaguchi for his constructive suggestions.

\medskip
\begin{flushleft}
Yuka Suzuki\\
Graduate School of Mathematics\\
Kyushu University\\
Nishi-ku, Fukuoka 819-0395 \\
Japan\\
y-suzuki@math.kyushu-u.ac.jp
\end{flushleft}

\end{document}